\documentclass[12pt]{amsart}
\usepackage[all]{xy}
\usepackage{a4wide}
\usepackage{amssymb}
\usepackage{amsthm,xcolor}
\usepackage{hyperref} 
\usepackage{graphicx}
\hypersetup{colorlinks=true,linkcolor=blue,citecolor=magenta}
\usepackage{amsmath}
\usepackage{amscd,enumitem}
\usepackage{verbatim}
\usepackage{float}
\usepackage{caption}
\usepackage{color}
\usepackage[mathscr]{eucal}
\usepackage[all]{xy}
\usepackage{hyperref}
\usepackage{bbm}
\usepackage{tikz}
\usetikzlibrary{quotes,angles}
\usepackage{dcolumn}
\numberwithin{equation}{section}

\def\={\;=\;}  \def\+{\,+\,} \def\m{\,-\,}

\theoremstyle{plain}
\newtheorem{theorem}{Theorem}
\numberwithin{theorem}{section}

\newtheorem*{corollary*}{Corollary}
\newtheorem*{Example*}{Example}

\newtheorem{lemma}[theorem]{Lemma}
\newtheorem{proposition}[theorem]{Proposition}

\newtheorem*{conjecture}{Conjecture}
\theoremstyle{definition}

\newtheorem*{def*}{Definition}
\newtheorem*{theorem*}{Theorem}

\newtheorem*{examples}{Examples}
\newtheorem*{definition*}{Definition}

\theoremstyle{remark}
\newtheorem*{remark}{Remark}

\newtheorem*{threeremarks}{Three remarks}

\newcommand{\maj}{{\text {\rm maj}}}

\newcommand{\len}{\mathrm{length}}
\newcommand{\R}{\mathbb{R}}

\newcommand{\Z}{\mathbb{Z}}

\newcommand{\C}{\mathbb{C}}
\newcommand{\PL}{\mathrm{PL}}

\newcommand{\binomial}[2]{\left ( \begin{matrix} #1\\ #2 \end{matrix}\right) }

\title[Tur\'an inequalities for the plane partition function]{Tur\'an inequalities for the plane partition function}

\author{Ken Ono, Sudhir Pujahari, and Larry Rolen}

\address{Department of Mathematics, University of Virginia, Charlottesville, VA 22904}
 \email{ken.ono691@virginia.edu}

\address{School of Mathematical Sciences, National Institute of Science Education and Research, Bhubaneswar, HBNI, P. O. Jatni, Khurda 752050, Odisha, India}
\email{spujahari@niser.ac.in}

\address{Department of Mathematics, Vanderbilt University, Nashville, TN 37240}
\email{larry.rolen@vanderbilt.edu}


\begin{document}
\thanks{The first
  author thanks  the Thomas Jefferson Fund and the NSF
(DMS-2002265 and DMS-2055118) for their generous support, as well as  the Kavli Institute grant NSF PHY-1748958.
This work was supported by a grant from the Simons Foundation (853830, LR). The third author is also grateful for support from a 2021-2023 Dean's Faculty Fellowship from Vanderbilt University. }

\begin{abstract} 
Heim, Neuhauser and Tr\"oger recently established some inequalities for MacMahon's plane partition function $\PL(n)$ that generalize known results for Euler's partition function $p(n)$.
They also conjectured that $\PL(n)$ is log-concave for all $n\geq 12.$ We prove this conjecture.
Moreover, for every $d\geq 1$, we prove their speculation that $\PL(n)$ satisfies the degree $d$ Tur\'an inequalities for sufficiently large $n$. The case where $d=2$ is the case of log-concavity.
\end{abstract}

\maketitle

\section{Introduction and Statement of Results}

A {\it partition} of a non-negative integer $n$ is any non-increasing sequence of positive integers that sum to $n$.
Hardy and Ramanujan famously proved that the partition function $p(n),$ which counts the number of integer partitions of $n,$ 
satisfies the asymptotic formula
$$
p(n)\sim \frac{1}{4n\sqrt{3}}\cdot e^{\pi\sqrt{2n/3}}.
$$
In the '70s, Nicolas \cite{Nicolas} employed such asymptotics to prove\footnote{This result was reproved in recent work by DeSalvo and Pak \cite{DP}.} that $p(n)$ is
log-concave for $n>25,$ where a sequence of real numbers $\{\alpha(0), \alpha(1),\dots\}$ is said to be {\it log-concave at $n$} if
\begin{equation}\label{logconcave}
\alpha(n)^2 \geq \alpha(n-1)\alpha(n+1).
\end{equation}

The condition of log-concavity for nonvanishing real sequences $\{\alpha(n)$\} is the first example of the {\it Tur\'an inequalities}, which can be conveniently formulated in terms of Jensen polynomials.
The  {\it Jensen polynomial of degree $d$ and shift $n$} is defined by
 \begin{equation}\label{JensenPolynomial}
J_\alpha^{d,n}(X)\,:=\,\sum_{j=0}^d \binomial dj\alpha(n+j)\,X^j.
\end{equation}
For degree $d=2$ and shift $n-1,$ the  roots are
$$
\frac{-\alpha(n)\pm \sqrt{\alpha(n)^2-\alpha(n-1)\alpha(n+1)}}{\alpha(n+1)}.
$$
Therefore, $\alpha(n)$ is log-concave at $n$ if and only if the roots of $J_{\alpha}^{2,n-1}(X)$ are real.
Generalizing log-concavity,  a real sequence is said to satisfy the
{\it degree $d$ Tur\'an inequalities at $n$} if $J_a^{d,n-1}(X)$ is hyperbolic, where a polynomial is hyperbolic if all of its roots are real.

There have been several recent works on the higher Tur\'an inequalities for $p(n).$
Chen, Jia and Wang \cite{ChenJiaWang} 
proved that  $J_p^{3,n}(X)$ is hyperbolic for $n\geq 94,$ which inspired them to conjecture,
for every integer $d\geq 1,$ that there is an integer $N_p(d)$
for which $J_p^{d,n}(X)$ is hyperbolic for $n \geq N_p(d)$.
Griffin, Zagier and two of the authors recently proved (see Theorem~5 of \cite{GORZ}) this conjecture for all degrees $d$. Furthermore, recent work
by Larson and Wagner \cite{LarsonWagner} established the
optimal values
$N_p(4)=206$ and $N_p(5)=381,$ as well as the effective bound $N_p(d)\leq (3d)^{24d}(50d)^{3d^2}.$

We consider such questions for {\it plane partitions} (for background, see  references by Andrews \cite{Andrews} and Stanley \cite{Stanley}). A {\it plane partition} of size $n$ is an array of non-negative integers $\pi:=(\pi_{i,j})$ for which $|\pi|:=\sum_{i,j} \pi_{i,j}=n$, in which the rows and columns are weakly decreasing. The figure below offers a 3d rendering of a plane partition.

\smallskip
\begin{center}
\includegraphics[height=40mm]{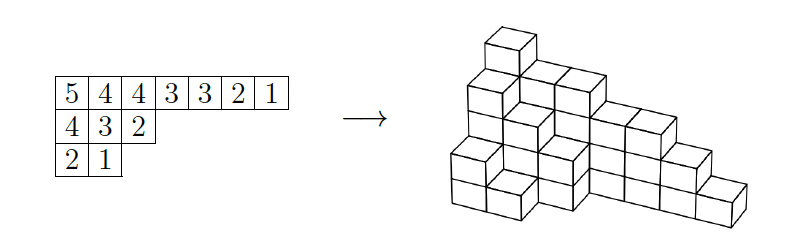}
\captionof{figure}{Example of a plane partition}
\end{center}
\smallskip
If $\PL(n)$ is the number of size $n$ plane partitions, then
MacMahon \cite{MacMahon} proved that
\begin{equation}\label{GenFcn}
f(x)=\sum_{n=0}^{\infty} \PL(n)x^n:=\prod_{n=1}^{\infty}\frac{1}{(1-x^n)^n}=1+x+3x^2+6x^3+13x^4+24x^5+48x^6+\dots.
\end{equation}
This function also appears prominently in physics in connection with the enumeration of small black holes in string theory. Indeed, $f(x)$ is the generating function (for example, see Appendix E of \cite{DDMP})
for the number of BPS bound states between a $D6$ brane and $D0$ branes on $\C^3.$

Heim, Neuhauser, and Tr\"oger \cite{HNT} have  undertaken a study of $\PL(n)$ in analogy with the aforementioned results  on $p(n)$. In addition to proving many inequalities satisfied by $\PL(n)$, they
pose  two further conjectures.  They prove (see Theorem 1.2 of \cite{HNT}) that $\PL(n)$ is log-concave for sufficiently large $n$, and they pose the following explicit conjecture (see Conjecture 1 of \cite{HNT}).

\begin{conjecture}[Heim, Neuhauser, and Tr\"oger]
The function $\PL(n)$ is log-concave for  $n\geq 12.$
\end{conjecture}

Here we resolve this problem.

\begin{theorem}\label{Thm1}
The Heim-Neuhauser-Tr\"oger Conjecture on the log-concavity of $\PL(n)$ is true.
\end{theorem}

\noindent
 Heim et al. also conjectured the direct analog of
 the Chen-Jia-Wang Conjecture on the higher degree Tur\'an inequalities.
 We prove this conjecture.

\begin{theorem}\label{Thm2}
If $d$ is a positive integer, then $J_{\PL}^{d,n}(X)$ is hyperbolic for all sufficiently large $n$.
\end{theorem}

\begin{remark}
In his Ph.D thesis \cite{Pandey}, Pandey will obtain an effective form of Theorem~\ref{Thm2} that is a counterpart to the bound of
$N_p(d)\leq (3d)^{24d}(50d)^{3d^2}$ established \cite{LarsonWagner} by Larson and Wagner for $p(n).$
\end{remark}

The proofs of Theorems~\ref{Thm1} and \ref{Thm2} require a strong asymptotic formula for $\PL(n).$   In the 1930s, Wright \cite{Wright} 
adapted the ``circle method'' of Hardy and Ramanujan to prove  asymptotic formulas for $\PL(n).$ He obtained such a formula for every positive integer $r$,
where the implied error terms are smaller with larger choices of $r$ for large $n$. 
The bulk of this paper is devoted to the lengthy and delicate task of obtaining the first
 asymptotics with explicitly bounded error terms.

To state these formulas, we require 
the two constants
\begin{equation}\label{constants}
A:=\zeta(3)\approx 1.202056\dots, \ \ \ \ {\text {\rm and}}\ \ \ \ 
 c:=2 \int_0^{\infty} \frac{y \log \, y }{e^{2 \pi y}-1} dy= \zeta'(-1)\approx -0.16542\dots.
 \end{equation}
Furthermore, for any pair of non-negative integers $s$ and $m$, we define coefficients $c_{s,m}(n)$ by
$$
 \frac{(1+y)^{2s+2m+\frac{13}{12}}}{(3+2y)^{(m+\frac{1}{2})}}=:\sum_{n=0}^{\infty}c_{s,m}(n)y^n.
$$
 In terms of these coefficients, we define the important numbers
 \begin{equation}\label{asm}
 b_{s,m}:=c_{s,m}(2m).
 \end{equation}
The asymptotic formulas we obtain are defined in terms of special numbers $\beta_0, \beta_1,\dots.$
To define them, for every positive integer $s$ we let
\begin{equation}\label{alphadefn} \alpha_s:=\frac{2 \Gamma(2s+2)\zeta(2s)\zeta(2s+2)}{s(2 \pi)^{4s+2}}.
\end{equation}
The real numbers $\beta_s$ are the Taylor coefficients of
\begin{equation}\label{BetaDefn}
\exp\left(-\sum_{i=1}^{\infty}\alpha_i y^i\right)=:\sum_{n=0}^{\infty} \beta_s y^s.
\end{equation}
 
For each $r,$ we use the numbers $\beta_0,\dots, \beta_{r+1}$ to derive the following explicit asymptotic formula.

\begin{theorem}\label{EffectiveWright}
If $r\in \Z^{+},$ then for every integer $n\geq  \max(n_r, \ell_r, 87)$ (see (\ref{nrdefn}-\ref{ellrdefn})) we have 
$$
\PL(n)=\frac{e^{c+3AN_n^2}}{2\pi }\sum_{s=0}^{r+1}\sum_{m=0}^{r+1}\frac{(-1)^m\beta_sb_{s,m}\Gamma\left(m+\frac12\right)}{A^{m+\frac12}N_n^{2s+2m+\frac{25}{12}}}+E_r^{\maj}(n)+E^{\min}(n),
$$
where  $|E_r^{\maj}(n)|\leq \widehat{E}_r^{\maj}(n)$  (see definition (\ref{ErMaj})), $N_n:=(\frac{n}{2A})^{\frac{1}{3}}$ and
$$|E^{\min}(n)|\leq \exp\left(\left(3A-\frac{2}{5}\right)n^2/(2A)^{\frac{2}{3}}\right).$$
\end{theorem}

\begin{threeremarks}\ \newline
\noindent
(1) The $s=m=0$ term in Theorem~\ref{EffectiveWright}  gives the well-known asymptotic\footnote{It is well-known that Wright has a typographical error where his asymptotic is off by a factor $\sqrt{3}.$}
$$
\PL(n) \sim \frac{(2^{25} A^7)^{\frac{1}{36}}e^c}{\sqrt{12 \pi}\cdot n^{\frac{25}{36}}} \exp\left(\sqrt[3]{\frac{27An^2}{4}}\right).
$$
Theorem~\ref{EffectiveWright} has two explicit error terms. The  error $E^{\min}(n),$ which is independent of $r$ and  arises from ``minor arc'' integrals, is exponentially smaller as $(3A-\frac{2}{5})/(2A)^{\frac{2}{3}}
\approx 1.79 < \sqrt[3]{27A/4} \approx 2.01.$  The error term $E_r^{\maj}(n)$, which arises from ``major arc'' integrals, only offers small power savings in $n$ that improve with larger choices of $r$. In addition to the expected complications required to make error terms explicit, the proof of Theorem~\ref{Thm1} is hampered by the small size of these power savings. This annoying problem does not arise\footnote{This comment for $p(n)$ applies for the Fourier coefficients of all non-positive weight weakly holomorphic modular forms.}
where working with further terms in the circle method gives exponentially improved error terms.
Since the effective bounds must be sufficient to reduce the conjecture to a finite range that can be handled by computer, the task of proving Theorem~\ref{Thm1} is a delicate balance between theory and practicality.
To prove Theorem~\ref{Thm1}, we use the case of $r=2$, where the major arc power savings is on the order of $n^{-\frac{7}{3}},$ and our theoretical bounds are sufficient  to confirm the conjecture
for all $n\geq 8820.$

%\smallskip
%\noindent
%(2) Wright's main asymptotic formula is presented as a single sum, as opposed to the double sum in
%$s$ and $m$ in Theorem~\ref{EffectiveWright}. This double sum formulation is the main device in  Wright's proof of his asymptotic formula.
%He makes a further simplifcation to obtain a single sum expression. We do not take this extra step as it would introduce further
% error.

 \smallskip
 \noindent
 (2) Almkvist \cite{Almkvist} and Govindarajan and Prabhakar \cite{GP} have refined Wright's asymptotic formula in a different way. 
 At the expense of requiring more summands as a function of $n$  (i.e. $\sim \kappa \sqrt[3]{n}$ many summands), their formulas give  precise asymptotics.  Given a positive integer $n$,  choosing $r\sim \kappa \sqrt[3]{n}$ in Theorem~\ref{EffectiveWright} gives similarly strong asymptotics, with the added benefit that the error terms are explicitly bounded.  
 
 \smallskip
 \noindent
 (3) Wright's main asymptotic formula is presented as a single sum, as opposed to the double sum in
$s$ and $m$ in Theorem~\ref{EffectiveWright}. This double sum formulation is the main device in  Wright's proof of his asymptotic formula.
He makes a further simplifcation to obtain a single sum expression. We do not take this extra step as it would introduce further
 error. 
\end{threeremarks}

\begin{examples}\label{IntroExample} 
For all $n\geq105$, the $r=1$ case of Theorem~\ref{EffectiveWright} implies (after some calculation) that
\begin{displaymath}
\begin{split}
\PL(n)=e^{3\cdot2^{-\frac23}A^{\frac13} n^{\frac23}}n^{-\frac{25}{36}} &\Big(\frac{2^{\frac{25}{36}}e^cA^{\frac{7}{36}}}{\sqrt{12\pi}}-\frac{\sqrt{3}\cdot2^{\frac{13}{36}}e^{c} {\left(3 A + 1385\right)} }{25920 \, \sqrt{\pi} A^{\frac{5}{36}}}n^{-\frac{2}{3}}\notag\\&  -\frac{\sqrt{3} \cdot2^{\frac{1}{36}} e^c{\left(1377 A^{2} - 370650  A+ 12525625\right)}}{1567641600  \sqrt{\pi} A^{\frac{17}{36}}}n^{-\frac{4}{3}} + E(n)\Big),\label{r1estimate}
\end{split}
\end{displaymath}
where $|E(n)|\leq 527n^{-\frac{5}{3}}.$
If $\widehat{\PL}_1(n)$ denotes this formula without the error $E(n),$  then we  have
$$E(n):=\left(\PL(n)-\widehat{\PL}_1(n)\right)\cdot e^{-3\cdot2^{-\frac23}A^{\frac13} n^{\frac23}}n^{\frac{25}{36}}.$$
 Table~\ref{table1}  illustrates the observed strength of  the power savings obtained by the $\widehat{\PL}_1(n)$ estimate.

\medskip
\begin{center} 
\begin{small}
\begin{tabular}{|c|c|c|c|}
 \hline
$n$ &$\PL(n)$& $E(n)$ & $527n^{-\frac53}$ \\ \hline \hline
$100$ & $5.92\ldots\times 10^{16}$ &  $-1.18\ldots\times 10^{-7}$ & $0.24\ldots$ \\ \hline
$200$ &$4.06\ldots\times 10^{27}$ & $-3.00\ldots\times 10^{-8}$ & $0.07\ldots$ \\ \hline
$\vdots$ & $\vdots$ & $\vdots$ & $\vdots$\\ \hline
$500$ & $2.91\ldots \times 10^{52}$ & $-4.87\ldots\times 10^{-9}$ & $0.01\ldots$ \\ \hline
\end{tabular}
\captionof{table}{Numerics for $r=1$ case of Theorem~\ref{EffectiveWright}}\label{table1}
\end{small}
\end{center}
\medskip

For $n\geq 87$, explicit calculations\footnote{We leave the details of the proof of the simpler
$r=1$ case  to the reader.} with the $r=2$ case of Theorem~\ref{EffectiveWright}
 gives (see \eqref{r2estimate})
$$\PL(n)=\widehat{\PL}_2(n)+E_2(n),$$
where $\widehat{\PL}_2(n)$ is defined by (\ref{Mainr2}), and 
$|E_2(n)|\leq  \mathcal{E}_2(n):=227e^{3AN_n^2}n^{-\frac{109}{36}}+e^{\left(3A-\frac25\right)N_n^2}.$
The bound $\mathcal{E}_2(n)$ is smaller than $\widehat{\PL}_2(n)$ for $n\geq96.$ 
 Table~\ref{table2} below illustrates the strength of this estimate.
\medskip
\begin{center} 
\begin{small}
\begin{tabular}{|c|c|c|c|}
 \hline
$n$ &$\widehat{\PL}_2(n)-\mathcal E_2(n)$& $\PL(n)$ & $\widehat{\PL}_2(n)+\mathcal E_2(n)$ \\ \hline \hline
$100$ & $5.932\ldots\times 10^{15}$ &  $5.920\ldots\times 10^{16}$ & $1.124\ldots\times 10^{17}$ \\ \hline
$200$ &$3.706\ldots\times 10^{27}$ & $4.066\ldots\times 10^{27}$ & $4.426\ldots\times 10^{27}$ \\ \hline
$\vdots$ & $\vdots$ & $\vdots$ & $\vdots$\\ \hline
$500$ & $2.913\ldots\times 10^{52}$ & $2.915\ldots\times 10^{52}$ & $2.917\ldots\times 10^{52}$ \\ \hline
$1000$ & $3.542\ldots\times 10^{84}$ & $3.542\ldots\times 10^{84}$ & $3.542\ldots\times 10^{84}$ \\ \hline
%$1500$ & $3.2586\ldots\cdot10^{111}$ & $3.2588\ldots\cdot10^{111}$ & $3.2589\ldots\cdot10^{111}$ \\ \hline
%$2000$ & $4.0051\ldots\cdot10^{135}$ & $4.0052\ldots\cdot10^{135}$ & $4.0053\ldots\cdot10^{135}$ \\ \hline
\end{tabular}
\captionof{table}{Numerics for $r=2$ case of Theorem~\ref{EffectiveWright}}\label{table2}
\end{small}
\end{center} 
\end{examples}
\medskip

 In Section~\ref{Wright} we prove Theorem~\ref{EffectiveWright} by modifying Wright's implementation of the circle method.
As is common for most applications of the circle method, the proof follows by considering integrals over ``major'' and ``minor'' arcs.
Our analysis of the major arc contributions is essentially a necessarily lengthy and careful refinement of Wright's original work.
However, our analysis of the minor arc contributions follows a completely different approach, which relies on work of Zagier and  a careful application of Euler-Maclaurin summation.
In Section~\ref{ProofsPart2}
we deduce Theorem~\ref{Thm1} and Theorem~\ref{Thm2} from Theorem~\ref{EffectiveWright}.
Theorem~\ref{Thm2} follows  from recent work
by Griffin, Zagier, and two of the authors in 
 \cite{GORZ}  on Jensen polynomials for suitable arithmetic sequences.

\section*{Acknowledgements} \noindent
The authors thank Kathrin Bringmann, Will Craig, Michael Griffin, Bernhard Heim, Greg Moore, Boris Pioline,  and Wei-Lun Tsai for useful comments. Finally, we are grateful to the anonymous referees for their careful reading of the manuscript and their suggestions.

\section{Effective form of Wright's asymptotic formulas}\label{Wright}

For convenience, we begin by outlining Wright's strategy for obtaining asymptotics for $\PL(n),$ which is a modification of the classical
``circle method.'' 
 For positive integers $n,$ we recall that $N_n=(\frac{n}{2A})^{\frac{1}{3}},$
where $A=\zeta(3)$ as in (\ref{constants}), 
 and we consider the circle
 \begin{equation}\label{circle}
 C_{N_n}:=\left \{x: |x| =e^{-\frac{1}{N_n}}\right \}.
 \end{equation}
Throughout, we let $\theta_x$ denote the principal value of $\arg(x),$ 
and we divide
$C_{N_n}$ into a ``major	 arc''  $C^{'}_{N_n},$ consisting of those $x$ with $|\theta_x|< \frac{1}{N_n},$
and the ``minor arc'' $C_{N_n}^{''}$ which is its complement. 
 
In terms of the generating function $f(x)$ in (\ref{GenFcn}), Cauchy's integral formula immediately gives
\begin{equation}\label{Cauchy}
\PL(n) = J(n)+E^{\min}(n),
\end{equation}
where
\begin{equation}\label{arcs}
J(n):= \frac{1}{2 \pi i} \int_{C_{N_n}^{'}} \frac{f(x)}{x^{n+1}}dx \ \ \ \ \ \ {\text {\rm and}}\ \ \ \ \ \
E^{\min}(n):= \frac{1}{2 \pi i} \int_{C_{N_n}^{''}} \frac{f(x)}{x^{n+1}}dx
\end{equation}
(i.e. with the usual counterclockwise orientation).
Wright analyzes $J(n)$ and $E^{\min}(n)$ separately\footnote{Wright referred to $J(n)$ as $J_1(n)$ (resp. $E^{\min}(n)$ as $J_2(n)$).}. 
The asymptotics arise from the major arc piece $J(n)$, and the minor arc $E^{\min}(n)$ piece
is a small error term.
To prove Theorem~\ref{EffectiveWright}, we improve on Wright's analysis of $E^{\min}(n)$ with a completely different argument, and we meticulously estimate $J(n)$ to obtain explicit estimates.

\subsection{Explicit bounds over the minor arcs}
Here we bound $E^{\min}(n)$  using a different method from that of Wright.
Instead of working directly with $f(x)$, we use its logarithmic derivative, which
 is the generating function for the sums of squares of divisors. Thanks to this interpretation, we make connection with work of  Zagier \cite{Zagier}, and we can then effectively bound $E^{\min}(n)$ using Euler-Maclaurin summation and calculus. 

\begin{proposition}\label{ZagierHelp}
For all $n\geq 87,$ we have
$|E^{\min}(n)|\leq \exp \left( \left(3A-\frac{2}{5}\right)\left(\frac{n}{2A}\right)^{\frac{2}{3}}\right).$
\end{proposition}

\subsubsection{Lemmata for Proposition~\ref{ZagierHelp}}
The next lemma bounds $\log f(|x|)$ on the minor arcs.

\begin{lemma}\label{Lemma1}
If $x\in C_{N_n}^{''},$ then we have
$\log f(|x|)\leq A{N_n}^2+0.33N_n-0.5.$
\end{lemma}

\begin{proof} 
We begin by noting that
\[
\log f(|x|)=\sum_{m\geq1}\frac{|x|^m}{m(1-|x|^m)^2},
\]
which is known as the {\it MacMahon function.}
We let $|x|=:q$ and $t:=1/N_n,$ and so we have $|x|=q=e^{-t}$. 
Taking the derivative, we obtain the Lambert series
\[
L(q):=q\frac{d}{dq}\sum_{m\geq1}\frac{q^m}{m(1-q^m)^2}=\sum_{m\geq1}\frac{q^m(1+q^m)}{(1-q^m)^3}. 
\]
Thanks to the elementary fact that
$X(1+X)/(1-X)^3=\sum_{k=1}^{\infty} k^2 X^k,$ we have that
\[
L(q)=g_3(q):=\sum_{m\geq1}\frac{m^2q^m}{1-q^m}. 
\]

As $q=e^{2\pi i\tau}$ with $\tau=\frac{it}{2\pi}$, we have
$q\frac{d}{dq}=\frac{1}{2\pi i}\frac{d}{d\tau}=-\frac d{dt},$  we have
$\frac{d}{dt}\log f(e^{-t})=-g_3(e^{-t}). $
Integrating this relation and adding the correct constant, we obtain 
\[
\log f(|x|)=\log f(e^{-t})=\int_t^1g_3(e^{-z})dz+\log f(e^{-1}). 
\]
Since $\log f(e^{-1})\approx 1.036$, we can bound this by 
\begin{equation}\label{lftog3}
|\log f(|x|)|=\log f(|x|)\leq \left|\int_t^1g_3(e^{-z})dz\right|+1.04. 
\end{equation}

Estimating the integral in this inequality is more involved.  We make use of Zagier's work \cite{Zagier} on generating functions of arbitrary divisor power sums.  He considers (see p. 15 of \cite{Zagier}) 
 the function $g_3(e^{-z})$ as $z\searrow0$.  In the $k=3$ case  of Example 3, he applies Proposition 3 of \cite{Zagier} with
$F(t):=\frac{t^2}{e^t-1}$ to obtain an asymptotic expansion for 
$g_3(e^{-t})=\frac{1}{t^2}\sum_{m\geq1}F(mt).$

As we need an estimate with explicitly bounded error, as opposed to an  asymptotic expansion, we dig into the proof of Proposition 3 of \cite{Zagier}. This gives, for each $k\geq1$, an exact formula for $g_3(e^{-t})$, where $k$ controls the number of terms in this asymptotic expansion, and leaves an integral that must be analyzed. 
The $k=1$ case\footnote{This is $N=1$ in Zagier's paper.} gives
\[
\sum_{m\geq1}F(mt)=\frac{1}{t}\int_0^{\infty}F(z)dz+\frac{(-1)^0B_1F(0)t^0}{1!}+(-t)^0\int_0^{\infty}\frac{F'(x)\overline{B_1}(x)}{1!}dx,
\]
where $B_1=-1/2$ is the Bernoulli number, and $\overline{B_1}(x)=x-\lfloor x\rfloor-\frac12$ is a periodization of the first Bernoulli polynomial $B_1(x)=x-\frac12$. Since $F(t)$ has a removable singularity at $t=0$ with Taylor expansion
$F(t)=t-\frac12 t^2+\ldots,$ we have
$F(0)=0.$ Moreover, Zagier computed that
\[
\int_0^{\infty}F(z)dz=(3-1)!\zeta(3)=2\zeta(3)=2A. 
\] 
Combining these observations, we obtain
\[
g_3(e^{-t})=\frac{2A}{t^3}-\frac{1}{t^2}\int_0^{\infty}F'(x)(x-\lfloor x\rfloor-1/2)dx. 
\]
As $|x-\lfloor x\rfloor-1/2|\leq\frac12$,  we have
\[
|g_3(e^{-t})|\leq \frac{2A}{t^3}+\frac{1}{2t^2}\int_0^{\infty}|F'(x)|dx
=
 \frac{2A}{t^3}+\frac{1}{2t^2}\int_0^{\infty}\left|\frac{(xe^x-2e^x+2)x}{(e^x-1)^2}\right|dx. 
\]
Since $\int_0^{\infty}|F'(x)| dx\approx 0.6471$, we have
$|g_3(e^{-t})|\leq 2A/t^3+0.33/t^2.$
Therefore, \eqref{lftog3} gives
\begin{equation*}
\begin{aligned}
\log f(|x|)&\leq \int_t^1\left(\frac{2A}{z^3}+\frac{0.33}{z^2}\right)dz +1.04
\leq\frac{A}{t^2}+\frac{0.33}{t}-0.5.
\end{aligned}
\end{equation*}
Letting $t=1/N_n$ gives the lemma.
\end{proof}

The proof of Proposition~\ref{ZagierHelp} also requires the following convenient lower bounds.

\begin{lemma}\label{Lemma2}
If $x\in C_{N_n}^{''},$ then we have
$$
\frac{|x|}{(1-|x|)^2}-\frac{|x|}{|1-x|^2}\geq\frac{N_n^2}2-\frac1{12}.
$$
\end{lemma}
\begin{proof}
We note that
\begin{equation}\label{elemweq}
\frac{|x|}{(1-|x|)^2}-\frac{|x|}{|1-x|^2}=\frac{|x|}{(1-|x|)^2}\left(1-\left(\frac{1-|x|}{|1-x|}\right)^2\right).
\end{equation}
We now estimate
\[
\frac{1-|x|}{|1-x|}=\frac{1-e^{-1/N_n}}{|1-x|}
\]
on $C_{N_n}''$. Recall that this is the arc of the circle of radius $e^{-1/N_n}$ with angles ranging from $1/N_n$ to $2\pi-1/N_n$. Geometrically, we have that the closest a point $x$ on this arc can get to the point $(1,0)$ in the plane is when the angle is $1/N_n$ (or $2\pi-1/N_n$), and so 
\[
\frac{1-|x|}{|1-x|}\leq \frac{1-e^{-1/N_n}}{|1-e^{-1/N_n}e^{i/N_n}|}.
\]
We obtain a formula from the denominator using the Law of Cosines.
Namely, if we draw a triangle (see Figure 2) with sides consisting of the line from $(0,0)$ to $x=e^{-1/N_n}e^{i/N_n}$ on the circle, the line from $(0,0)$ to $(1,0)$, and the line connecting $x$ to $(1,0)$, then the unknown length $|1-x|$ is the opposite side of the angle $1/N_n$ bounded by side lengths $1$ and $e^{-1/N_n}$. 

\smallskip
\smallskip
\begin{center}
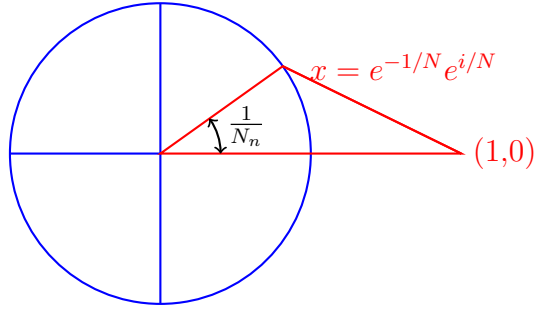

\begin{tikzpicture}
\draw[blue,  thick] (0,0) circle (2);
\draw[-, blue, thick] (-2,0) -- (0,0);
\draw[-, blue, thick] (0,-2) -- (0,2);
\draw[red, thick] (0,0) coordinate (a) 
-- (1.627466,1.162476) coordinate (c) 
-- (4,0) coordinate (b) node[right] {(1,0)}
-- cycle;
\draw[red, thick] (1.627466,1.162476) coordinate (x) node[right]{\ \,$x=e^{-1/N}e^{i/N}$}
-- (4,0) 
pic["$\tiny{{\color{black}\frac{1}{{N_n}}}}$",draw=black,<->,angle eccentricity=1.5,angle radius=.8cm] {angle= b--a--c};
\end{tikzpicture}\\ \vskip.1in
\captionof{figure}{Triangle used to evaluate $|1-x|$}
\end{center}
\medskip
After letting $t=1/N_n,$ these facts imply that
\[
|1-e^{-1/N_n}e^{i/N_n}|=\sqrt{1+e^{-2/N_n}-2e^{-1/N_n}\cos(1/N_n)}=\sqrt{1+e^{-2t}-2e^{-t}\cos(t)}.
\]
Thus, we get that
\[
h(t):=\frac{1-e^{-1/N_n}}{|1-e^{-1/N_n} e^{i/N_n}|}=\frac{1-e^{-t}}{\sqrt{1+e^{-2t}-2e^{-t}\cos(t)}}=\frac1{\sqrt2}+\frac{\sqrt2}{48}t^2+\ldots.
\]
Since $t=1/N_n=(2A/n)^{\frac13}$ and $n\geq1$, we have that $t\leq(2A)^{\frac13}\leq1.34$. On the interval $[0,1.34]$, the maximum absolute value of $h''(t)$ is $\lim_{t\rightarrow0}h''(t)=\frac{\sqrt{2}}{24}$. 
Thus, by Taylor's Theorem, 
\[
|h(t)|=\frac1{\sqrt2}+O_{\leq}\left(\frac{\sqrt2}{48}t^2\right),
\]
where  $O_{\leq}(\cdot)$ means that the expression is bounded by $\cdot$ in absolute value (i.e. the implied constant can be chosen to be 1 with ordinary $O$-notation). 
Hence, for $x\in C_{N_n}''$, we have 
\begin{equation}\label{Lemma2Factor}
\frac{1-|x|}{|1-x|} \leq \frac1{\sqrt2}+O_{\leq}\left(\frac{\sqrt2}{48N_n^2}\right).
\end{equation}
Returning to \eqref{elemweq}, we estimate 
\[
\frac{|x|}{(1-|x|)^2}=
\frac{e^{-1/N_n}}{(1-e^{-1/N_n})^2}
=\frac{e^{-t}}{(1-e^{-t})^2}=t^{-2}-\frac1{12}+\frac{1}{240}t^2+\ldots
\]
By a similar use of Taylor's Theorem, we have the strict inequality 
\[
\frac{e^{-t}}{(1-e^{-t})^2} > t^{-2}-\frac1{12}=N_n^2-\frac1{12}.
\]
Using this bound, (\ref{elemweq}) and (\ref{Lemma2Factor}) completes the proof as
\begin{displaymath}
\begin{aligned}
\frac{|x|}{(1-|x|)^2}-\frac{|x|}{|1-x|^2}&\geq\left(N_n^2-\frac1{12}\right)\left(1-\left(\frac1{\sqrt2}+O_{\leq}\left(\frac{\sqrt2}{48N_n^2}\right)\right)^2\right)
\geq\frac{N_n^2}2-\frac1{12}.
\end{aligned}
\end{displaymath}
\end{proof}

\subsubsection{Proof of Proposition~\ref{ZagierHelp}}
To estimate
\[
E^{\min}(n)= \frac{1}{2 \pi i} \int_{C_{N_n}^{''}} \frac{f(x)}{x^{n+1}}dx,
\]
 we bound $|f(x)/x^{n+1}|$. Following Wright (see page 184 of \cite{Wright}), we note that 
\begin{equation}\label{flfelem}
|\log f(x)|\leq \log f(|x|)-\left(\frac{|x|}{(1-|x|)^2}-\frac{|x|}{|1-x|^2}\right).
\end{equation}
Lemma~\ref{Lemma1} proved that
$\log f(|x|)\leq A{N_n}^2+0.33N_n-0.5.$
Furthermore, Lemma~\ref{Lemma2} establishes that
$\frac{|x|}{(1-|x|)^2}-\frac{|x|}{|1-x|^2}\geq\frac{N_n^2}2-\frac1{12}.$
Therefore, \eqref{flfelem} gives
\begin{equation*}
\begin{aligned}
|\log f(x)|\leq AN_n^2+0.33N_n-0.5-\left(\frac{N_n^2}2-\frac1{12}\right)\leq(A-1/2)N_n^2+0.33N_n,
\end{aligned}
\end{equation*}
and so
$|f(x)|\leq e^{(A-1/2)N_n^2+0.33N_n}.$
Thus, the integrand in $E^{\min}(n)$ is bounded by 
\[
\frac{|f(x)|}{|x|^{n+1}}\leq e^{(A-1/2)N_n^2+0.33N_n+(n+1)/N_n}
=e^{(3A-1/2)N_n^2+0.33N_n+1/N_n}.
\]
Since the integral defining $E^{\min}(n)$  is along a curve of length bounded by the circumference of the whole circle of radius $e^{-1/N_n}$, which is $2\pi e^{-\frac1N_n}$, we finally find that
\[
 |E^{\min}(n)| \leq e^{(3A-1/2)N_n^2+0.33N_n+1/N_n-1/N_n}=e^{(3A-1/2)N_n^2+0.33N_n}.
\]
The claimed inequality for $n\geq 87$ follows  by analyzing this last expression. 

\subsection{Explicit major arc formulas}

The size of $\PL(n)$ is given by the major arc integral $J(n).$ 
To reduce the complexity of error terms,  for each positive integer $r$ we define thresholds
\begin{equation}\label{nrdefn}
n_r:=\min\left\{n\geq 1\ : \  0.056\cdot \sum_{s=1}^{r+1}\left(\frac{s\cdot A^{\frac{1}{3}}}{2^{\frac{7}{6}}n^{\frac{1}{3}}}\right)^{2s}\left(\frac{\pi^2 
n^{\frac{1}{3}}}{(2A)^{\frac{1}{3}}s}+2\right)< 1\right\}
\end{equation}
and 
\begin{equation}\label{ellrdefn}
\ell_r:=\min\left\{n\geq1\ : \ 2^{r+4} \pi^3\alpha_{r+2} N_n^{-2r-4}+5e^{-4.7N_n}<\frac12 \right\}.
\end{equation}

\begin{remark} The thresholds $\ell_r$ and $n_r$ are simple to compute.  For instance, we have that
$\ell_j=1$ for $ j\leq 22$ (resp. $ \ell_j=2$  for $23\leq j\leq 30$),
and
$n_1=1,$ $n_2=2,\dots, n_5=18.$

\end{remark}

The following explicit major arc estimate is the main result of this subsection.

\begin{proposition}\label{MajorArc}
If $r\in \Z^{+},$ then for every $n\geq \max(\ell_r,n_r, 55)$ we have 
$$
J(n)=\frac{e^{c+3AN_n^2}}{2\pi }\sum_{s=0}^{r+1}\sum_{m=0}^{r+1}\frac{(-1)^m\beta_sb_{s,m}\Gamma\left(m+\frac12\right)}{A^{m+\frac12}N_n^{2s+2m+\frac{25}{12}}}+E_{r}^{\maj}(n),
$$
where $|E_r^{\maj}(n)|\leq  \widehat{E}_r^{\maj}(n)$ (see (\ref{ErMaj})).
\end{proposition}

\subsubsection{Lemmata for Proposition~\ref{MajorArc}}

For fixed $n,$  Wright sets
\begin{equation}\label{z}
z=\log\left(\frac1x\right)=\log\left|\frac1x\right|-i\vartheta=\log\left (e^{\frac1N_n}\right)-i\vartheta=\frac{1}{N_n}-i\vartheta=:\rho e^{i\phi},
\end{equation}
and he defines
\begin{equation}\label{w}
w:=\mathrm{Re}\left(\frac{\pi}{2z}\right)=\frac{\pi \cos\phi}{2\rho}. 
\end{equation} 
On the major arc $C_{N_n}'$, we have $|\vartheta|<\frac1N_n$, and so 
\begin{equation}\label{rho}
\rho=\sqrt{N_n^{-2}+\vartheta^2}\in\left[\frac{1}{N_n},\frac{\sqrt2}{N_n}\right].
\end{equation}
Furthermore, we have
\begin{equation}\label{phi}
|\phi|=|\arctan(\vartheta N_n)|\leq\arctan(1)=\frac{\pi}{4}. 
\end{equation}

Wright uses the basic integral identity (see (3.17) of \cite{Wright})
\[
\int_0^{\infty}t\log(1-e^{-tz})dt=-\frac{A}{z^2}, 
\]
which implies 
\[
-\log f(x)+\frac{A}{z^2}=-\log f(x)-\int_0^{\infty}t\log(1-e^{-tz})dt.
\]
Using the generating function for $f(x)$, this becomes
\[
\sum_{m\geq1}m\log(1-e^{-mz})-\int_0^{\infty}t\log(1-e^{-tz})dt.
\]
Now since 
$\frac{1}{e^{2\pi i t}-1}-\frac{1}{1-e^{-2\pi it}}=-1,$ 
this can be written as 
$$
-\log f(x)+\frac{A}{z^2}=\sum_{m\geq1}m\log(1-e^{-mz})+
\int_{\Gamma}\left(\frac{t\log(1-e^{-tz})}{e^{2\pi i t}-1}-\frac{t\log(1-e^{-tz})}{1-e^{-2\pi i t}}\right)dt,
$$
where $\Gamma$ is the path from $0$ to $\infty$ which travels along the real axis, apart from sufficiently small semicircles at the positive integers above the real axis to avoid the poles of the integrand. 
\smallskip
\smallskip
\begin{center}
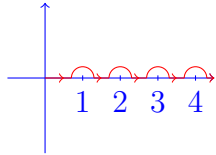

\begin{tikzpicture}
\draw[->,blue] (-0.5,0) -- (2.25,0) coordinate (x axis);
\draw[->,blue] (0,-1) -- (0,1) coordinate (y axis);
\draw[blue] (0.5 cm,1pt) -- (0.5 cm,-1pt) node[anchor=north,fill=white] {$1$};
\draw[blue] (1 cm,1pt) -- (1 cm,-1pt) node[anchor=north,fill=white] {$2$};
\draw[blue] (1.5 cm,1pt) -- (1.5 cm,-1pt) node[anchor=north,fill=white] {$3$};
\draw[blue] (2 cm,1pt) -- (2 cm,-1pt) node[anchor=north,fill=white] {$4$};
\draw[red] (0.65,0)  arc[radius = 1.5mm, start angle= 0, end angle= 180];
\draw[red] (1.15,0)  arc[radius = 1.5mm, start angle= 0, end angle= 180];
\draw[red] (1.65,0)  arc[radius = 1.5mm, start angle= 0, end angle= 180];
\draw[red] (2.15,0)  arc[radius = 1.5mm, start angle= 0, end angle= 180];
\draw[->, red] (0,0) -- (0.25,0);
\draw[-, red] (0.23,0) -- (0.35,0);
\draw[->,red] (0.65,0)--(0.75,0);
\draw[-,red](0.73,0)--(0.85,0);
\draw[->,red] (1.15,0)--(1.25,0);
\draw[-,red] (1.23,0)--(1.35,0);
\draw[->,red] (1.65,0)--(1.75,0);
\draw[-,red] (1.73,0)--(1.85,0);
\draw[->,red] (2.15,0)--(2.25,0);
\end{tikzpicture}\\ \vskip.05in
\captionof{figure}{The path of integration $\Gamma$}
\end{center}
\medskip

 Letting $\Gamma'$ be the reflection of the path $\Gamma$ across the $y$-axis, 
using  the Residue Theorem (note that the residue of $\frac{t\log(1-e^{-tz})}{e^{2\pi i t}-1}$ at $t=m\in\Z$ is $m\log(1-e^{-mz})/(2\pi i)$,  
Wright expresses\footnote{We note that Wright's notation does not clearly indicate that this is an exact identity.} (see (3.18) of \cite{Wright})  this as 
\begin{equation}\label{important}
-\log f(x)+\frac{A}{z^2}=
\int_{\Gamma'}\frac{t\log(1-e^{-tz})}{e^{2\pi i t}-1}dt
-\int_{\Gamma}\frac{t\log(1-e^{-tz})}{1-e^{-2\pi i t}}dt.
\end{equation}
To obtain an effective estimate for $f(x)$, we study these two integrals. 

\begin{lemma}\label{LemmaGamma} Assuming the notation and hypotheses above, we have
\begin{displaymath}
\begin{split}
\int_{\Gamma}\frac{t\log(1-e^{-tz})}{1-e^{-2\pi i t}}dt&=\int_0^{iw}\frac{t\log(1-e^{-tz})}{1-e^{-2\pi i t}}dt+O_{\leq}\left(35N_n^2e^{-\pi^2N_n}\right),\\
\int_{\Gamma^{'}} \frac{t \log \, (1-e^{-tz})}{e^{2 \pi i t}-1} dt&=\int_0^{-iw}\frac{t\log(1-e^{-tz})}{e^{2\pi i t}-1}dt+O_{\leq}\left(N_n^2e^{-\pi^2N_n}\right).
\end{split}
\end{displaymath}
\end{lemma}
\begin{proof}
By Cauchy's Theorem, Wright showed that the integral over $\Gamma$
(see p. 182 of \cite{Wright}) is 
\begin{equation}\label{intgamma}
\int_{\Gamma}\frac{t\log(1-e^{-tz})}{1-e^{-2\pi i t}}dt=\int_0^{iw}\frac{t\log(1-e^{-tz})}{1-e^{-2\pi i t}}dt+\int_{iw}^{iw+\infty}\frac{t\log(1-e^{-tz})}{1-e^{-2\pi i t}}dt,
\end{equation}
where in the second integral on the right hand side the path is horizontal and parallel to the $x$-axis. 
To estimate the absolute value of the integrand, note that by the reverse triangle inequality (letting $x\in(0,\infty)$ so that $iw+x$ denotes a typical point on the path of integration)
\[
\left|\frac{1}{1-e^{-2\pi i (iw+x)}}\right|\leq\left|\frac{1}{1-|e^{-2\pi i(iw+x)}|}\right|=\left|\frac{1}{1-e^{\pi^2\cos\phi/\rho}}\right|.
\]
Using \eqref{rho}, this is bounded by
$\left|1/(1-e^{\pi^2N_n})\right|\leq1.007e^{-\pi^2N_n}.$
To derive this bound, we used the fact that  $N_n\geq N_1\approx0.75$  is monotonically increasing in $n$ and 
$e^{-\pi^2 N_1} |1-e^{\pi^2N_1}|\approx 0.99936.$
Therefore, we have
\[
\left|\int_{iw}^{iw+\infty}\frac{t\log(1-e^{-tz})}{1-e^{-2\pi i t}}dt\right|\leq1.007e^{-\pi^2N_n}\int_{iw}^{iw+\infty}\left|t\log(1-e^{-tz})\right|dt.
\]
Following Wright again, and making the change of variables $v=tz$, this is bounded by 
\begin{equation}
\label{IntegralonL}
\begin{aligned}
1.007e^{-\pi^2N_n}\int_{L}\frac{\left|t\log(1-e^{-v})\right|}{|z|^2} dv
 \leq1.007N_n^2e^{-\pi^2N_n}\int_{L}\left|v\log(1-e^{-v})\right|dv =: U(n),
\end{aligned}
\end{equation}
where $L$ is the ray from $v=\frac{\pi i}{2}e^{i\phi}\cos\phi$ that forms the angle $\phi$ with the (positive) real axis. 

We split the integral in \eqref{IntegralonL} into two pieces. Wright also does this, but we make a slightly different choice below to assist us in our goal of obtaining effective estimates.  Throughout, let $v_t:=iwz+te^{i\phi}$, so that 
$L=\{v_t \ : \  t\in[0,\infty)\}.$
Noting that 
\begin{equation}\label{wzform}
wz=\frac{\pi\cos\phi}{2\rho}\cdot\rho e^{i\phi}=\frac{\pi \cos\phi e^{i\phi}}2,
\end{equation}
we first estimate the piece $|\log(1-e^{-v_t})|$ in the integrand of \eqref{IntegralonL}.
We use (4.5.6) of DLMF \cite{NIST}, which states that for complex arguments $y$, to obtain
\begin{equation}\label{NIST456}
|\log(1+y)|\leq-\log(1-|y|),\quad\text{ when }|y|<1.
\end{equation}
We will break up the line $L$ into a compact piece $L_1$ and a remaining piece $L_2,$ where we can utilize this bound. To see where it applies, we compute
\begin{equation}\label{evt}
|e^{-v_t}|=e^{-\operatorname{Re}(v_t)}=e^{w\operatorname{Im}(z)-\operatorname{Re}(te^{i\phi})}=e^{-w\vartheta-t\cos\phi}=e^{-\cos\phi\left(\frac{\pi\vartheta}{2\rho}+t\right)}.
\end{equation}
In turn, this is less than or equal to one if and only if 
$\cos\phi\left(\frac{\pi\vartheta}{2\rho}+t\right)\geq0.$
Using \eqref{rho} and \eqref{phi}, and using  $\vartheta\geq-1/N$,
we have 
\[
\cos\phi\left(\frac{\pi\vartheta}{2\rho}+t\right)>\frac{\sqrt2}{2}\left(-\frac{\pi}{2\sqrt2}+t\right)=-\frac{\pi}{4}+\frac{t}{\sqrt{2}},
\]
which gives
\begin{equation}\label{evtbound}
|e^{-v_t}|\leq e^{\frac{\pi}{4}-\frac{t}{\sqrt{2}}}.
\end{equation}
Hence, if we let
\[
L_1:=\left\{v_t\colon t\in[0,\pi/\sqrt8+1)\right \},\quad L_2:=\left \{v_t\colon t\in[\pi/\sqrt8+1,\infty)\right\},
\]
then we can use (\ref{NIST456}) to estimate the integrand on $L_2.$ Thanks to \eqref{evtbound}, for $t>\pi/\sqrt8$ we have
\[
|v_t\log(1-e^{-v_t})|\leq-|v_t|\cdot\log\left(1-|e^{-v_t}|\right)\leq-|v_t|\log\left(1-e^{\frac{\pi}{4}-\frac{t}{\sqrt2}}\right).
\]
Noting that 
$|v_t|\leq|\cos\phi|\frac{\pi}{2}+t\leq\frac{\pi}2+t,$
and using the bound $-\log(1-x)<x/(1-x)$ for $0\neq x<1$ on real-valued logarithms from (4.5.2) of DLMF \cite{NIST}, we  find
\[
\left|v_t\log(1-e^{-v_t})\right|\leq
\frac{\left(\frac{\pi}2+t\right)\cdot e^{\frac{\pi}{4}-\frac{t}{\sqrt{2}}}}{1-e^{\frac{\pi}4-\frac{t}{\sqrt{2}}}}.
\]

Since $L_1$ is compact, we then find the following estimate for \eqref{IntegralonL} (note that when we integrate on $L_2$, since we are integrating absolute values, the change of variables in the differential goes away as it has absolute value 1):
\begin{equation*}
\begin{aligned}
&U(n)=1.007N_n^2e^{-\pi^2N_n}\int_{L_1}\left|v\log(1-e^{-v})\right|dv
+1.007N_n^2e^{-\pi^2N_n}{N_n^2}\int_{L_2}\left|v\log(1-e^{-v})\right|dv
\\
&\ \ \ \ \leq
1.007N_n^2e^{-\pi^2N_n}\left(\frac{\pi}{\sqrt8}+1\right)\cdot\max\left\{|v_t|\cdot|\log(1-e^{-v_t})|\ \colon t\in[0,\pi/\sqrt8+1)\right\}
\\
&
 \ \ \ \ \ \ \ \ \ \ +1.007N_n^2e^{-\pi^2N_n}\int_{\frac{\pi}{\sqrt8}+1}^{\infty}\frac{\left(\frac{\pi}2+t\right)\cdot e^{\frac{\pi}{4}-\frac{t}{\sqrt{2}}}}{1-e^{\frac{\pi}4-\frac{t}{\sqrt{2}}}}dt
\\
&\ \ \ \ \leq1.007N_n^2e^{-\pi^2N_n}\left(\frac{\pi}{\sqrt8}+1\right)\cdot\max\left\{|v_t|\cdot|\log(1-e^{-v_t})| \colon t\in[0,\pi/\sqrt8+1)\right\}+4.8N_n^2e^{-\pi^2N_n}.
\end{aligned}
\end{equation*}

Now we estimate $\log(1-e^{-v_t})$ on for $t\in[0,\pi/\sqrt{8}+1)$.  We begin  by 
recalling that for complex $y$, the principal branch of the logarithm is given by 
$\log(y)=\log(|y|)+i\arg(y).$
Thus, we have
\begin{equation}\label{LogFinalBound}
|\log(1-e^{-v_t})\left|\leq|\log|1-e^{-v_t}|\right|+\pi.
\end{equation}
To bound the logarithm, we find the maximum and minimum on the interval $t\in[0,\pi/\sqrt8+1)$. 
The only critical point of $|1-e^{-v_t}|$ is at 
$t=\cosh(i\phi). $
Thus, the potential extrema of $|1-e^{-v_t}|$ are at $0$, $\pi/\sqrt8$, and $\cosh(i\phi)$. 
That is, the maximum of $|\log|1-e^{-v_t}||$ is bounded by 
$$\max\left\{|1-e^{-v_0}|, \ |1-e^{-v_{\pi/\sqrt8+1}}|, \ |1-e^{-v_{\cosh(i\phi)}}|\right\}.$$
We evaluate these in turn. 
For instance, we have that 
$|1-e^{-v_0}|=\left|1-e^{-\frac{\pi i\cos\phi e^{i\phi}}2}\right|.$
On the interval $\phi\in[-\pi/4,\pi/4]$, we have
$0.75\leq |1-e^{-v_0}|\leq 1.85$ and
$0.9\leq |1-e^{-v_{\pi/\sqrt8+1}}|\leq 1.4.$
Similarly, we find
$0.7\leq\left|1-e^{-v_{\cosh(i\phi)}}\right|\leq 1.45.$
Combining these observations, we have
$\left|\log|1-e^{-v_t}|\right|\leq0.615\ldots$
on $[0,\pi/\sqrt8+1)$. 
Thus, by \eqref{LogFinalBound}, we have
$|\log(1-e^{-v_t})|\leq3.76.$
To bound $|v_t|$ on this interval, using \eqref{wzform} and the triangle inequality, we find that
$|v_t|\leq\frac{\pi\cos\phi}{2}+t\leq\frac{\pi}{2}+\frac{\pi}{\sqrt8}+1=3.68\ldots.$
Therefore, we conclude that
\begin{equation*}
\begin{aligned}
&1.007N_n^2e^{-\pi^2N_n}\left(\frac{\pi}{\sqrt8}+1\right)\max\left\{|v_t|\cdot|\log(1-e^{-v_t})\colon t\in[0,\pi/\sqrt8+1)\right\}
\leq30N_n^2e^{-\pi^2N_n}.
\end{aligned}
\end{equation*}
As a consequence, we obtain
\[
1.007N_n^2e^{-\pi^2N_n}\int_{L}\left|v\log(1-e^{-v})\right|dv\leq35N_n^2e^{-\pi^2N_n}.
\]

Returning to \eqref{intgamma}, we have shown that 
\begin{equation}\label{intgammaest}
\int_{\Gamma}\frac{t\log(1-e^{-tz})}{1-e^{-2\pi i t}}dt=\int_0^{iw}\frac{t\log(1-e^{-tz})}{1-e^{-2\pi i t}}dt+O_{\leq}\left(35N_n^2e^{-\pi^2N_n}\right).
\end{equation}
This bounds the second integral in (\ref{important}), which is the first claim in the lemma.

The second claim in the lemma follows by
arguing as above (after suitable sign changes in the integrand) using the path of integration along $\Gamma'.$  Namely, we get
\begin{equation}\label{intgammapest}
\int_{\Gamma^{'}} \frac{t \log \, (1-e^{-tz})}{e^{2 \pi i t}-1} dt=\int_0^{-iw}\frac{t\log(1-e^{-tz})}{e^{2\pi i t}-1}dt+O_{\leq}\left(N_n^2e^{-\pi^2N_n}\right).
\end{equation}
\end{proof}

We require  bounds for  three  integrals,  two of which make use of the $\alpha_s$ defined by
(\ref{alphadefn}). 

\begin{lemma}\label{I_integralbounds}
Assuming the notation and hypotheses above, the following are true.

\noindent
(1) If $n$ is a positive integer, then we have
$$
\mathcal{I}_1:=2\int_0^w\frac{y\log(yz)}{e^{2\pi y}-1}dy=c+\frac{\log z}{12}+O_{\leq}\left(3e^{-4.7N_n}\right).
$$

\noindent
(2) If $n\geq n_r$ is an integer, then we have
$$
\mathcal{I}_2:=-2\sum_{s=1}^{r+1}\frac{\zeta(2s)z^{2s}}{s(2\pi)^{2s}}\int_0^w\frac{y^{2s+1}}{e^{2\pi y}-1}dy= -\sum_{s=1}^{r+1}\alpha_sz^{2s}+O_{\leq} \left(e^{-\frac{\pi^2 N_n}{2}}\right).
$$

\noindent
(3) If $n$ is a positive integer, then we have
$$
\mathcal{I}_3:=
-2\sum_{s\geq r+2}\frac{\zeta(2s)z^{2s}}{s(2\pi)^{2s}}\int_0^w\frac{y^{2s+1}}{e^{2\pi y}-1}dy\leq 2^{r+4} \pi^3\alpha_{r+2} N_n^{-2r-4}.
$$
\end{lemma}
\begin{proof}
We estimate these three quantities one-by-one.
For $\mathcal I_1,$ we have 
\[
\mathcal I_1=2\int_0^w\frac{y\log(yz)}{e^{2\pi y}-1}dy=2\int_0^\infty\frac{y\log(yz)}{e^{2\pi y}-1}dy-2\int_w^\infty\frac{y\log(yz)}{e^{2\pi y}-1}dy
=c+\frac{\log z}{12}-2\int_w^\infty\frac{y\log(yz)}{e^{2\pi y}-1}dy.
\]
To evaluate the last integral, we first note that 
$w=\frac{\pi\cos\phi}{2\rho}\geq\frac{\pi\sqrt2}{4\rho}\geq\frac{\pi N_n}{4}\geq0.58,$
where we used \eqref{phi} and  \eqref{rho} and the fact that $n\geq1$. By direct manipulation, for $y\geq0.58$, we have 
\begin{equation}\label{ExpEst}
\frac{1}{e^{2\pi y}-1}\leq1.1\cdot e^{-2\pi y}.
\end{equation}
Then on the interval $w\geq0.58$, we find that 
\begin{equation*}
\begin{aligned}
&\left|-2\int_w^\infty\frac{y\log(yz)}{e^{2\pi y}-1}dy\right|
\leq 2.2\left|\int_w^{\infty}y\log(yz)e^{-2\pi y}dy\right|\\
&\ \ \ \ \ \ \ \ \ \ \ \ =\frac{2.2}{4\pi^2}\cdot e^{-2\pi w}\left|1-e^{2\pi w}\operatorname{Ei}(-2\pi w)+(1+2\pi w)\log(wz)\right|,
\end{aligned}
\end{equation*}
where $\operatorname{Ei}(x):=-\int_{-x}^{\infty}\frac{e^{-t}dt}{t}$. 
Straightforward manipulation then gives
\begin{equation}\label{Splitting}
\left|-2\int_w^\infty\frac{y\log(yz)}{e^{2\pi y}-1}dy\right|\leq e^{-2\pi w}\left(1+0.056\cdot(1+2\pi w)(|\log \rho|+\pi)\right).
\end{equation}
If $n\geq7$, then $\sqrt 2/N_n<1$, and so $w\geq \pi N_n/4.$ Therefore,  \eqref{Splitting} is bounded from above by
\[
e^{-\frac{\pi^2N_n}{2}}\left(1+0.056\cdot\left(1+\frac{\pi^2N_n}{2}\right)(\log N_n+\pi)\right).
\]
It is straightforward to show that this is less than or equal to 
$3e^{-4.7N_n}.$
By direct computation for $1\leq n\leq6$ using \eqref{Splitting}, we find that this bound holds in general, and so we have
\begin{equation}\label{Dec16Eqn}
\mathcal I_1=c+\frac{\log z}{12}+O_{\leq}\left(3e^{-4.7N_n}\right).
\end{equation}

 Using \eqref{ExpEst} and the integral representation of $\zeta(s),$ we can manipulate $\mathcal{I}_2$ to obtain
\begin{equation*}
\begin{aligned}
\mathcal I_2&=-2\sum_{s=1}^{r+1}\frac{\zeta(2s)z^{2s}}{s(2\pi)^{2s}}\int_0^\infty\frac{y^{2s+1}}{e^{2\pi y}-1}dy+2\sum_{s=1}^{r+1}\frac{\zeta(2s)z^{2s}}{s(2\pi)^{2s}}\int_w^\infty\frac{y^{2s+1}}{e^{2\pi y}-1}dy
\\
&=
-2\sum_{s=1}^{r+1}\frac{\zeta(2s)z^{2s}}{s(2\pi)^{2s}}\int_0^\infty\frac{y^{2s+1}}{e^{2\pi y}-1}dy
+O_{\leq}\left(
2.2\cdot \sum_{s=1}^{r+1}\frac{\zeta(2s)z^{2s}}{s(2\pi)^{2s}}\int_w^\infty y^{2s+1}e^{-2\pi y}dy
\right)
\\
&=
-2\sum_{s=1}^{r+1}\frac{\zeta(2s)z^{2s}}{s(2\pi)^{2s}}\left(\frac{\Gamma(2+2s)\zeta(2+2s)}{(2\pi)^{2+2s}}\right)
+O_{\leq}
\left(
2.2\cdot \sum_{s=1}^{r+1}\frac{\zeta(2s)z^{2s}}{s(2\pi)^{2s}}\left(\frac{\Gamma(2+2s;2\pi w)}{(2\pi)^{2+2s}}\right)
\right),
\end{aligned}
\end{equation*}
where $\Gamma(a;x):=\int_x^{\infty}t^{a-1}e^{-t}dt$ is the {\it incomplete Gamma function}.
Using \eqref{alphadefn}, this is
\[
\mathcal{I}_2=-\sum_{s=1}^{r+1}\alpha_sz^{2s}+O_{\leq}
\left(
2.2\sum_{s=1}^{r+1}\frac{\zeta(2s)z^{2s}}{s(2\pi)^{2s}}\left(\frac{\Gamma(2+2s;2\pi w)}{(2\pi)^{2+2s}}\right)
\right).
\]
To estimate this, we note that the proof of Lemma~2.2 of \cite{CFT} shows, for $a>2,$ that
\[
\Gamma(a;x)\leq \frac{(x+b_a)^a-x^a}{ab_a}e^{-x},
\]
where $b_a:=\Gamma(a+1)^{\frac{1}{a-1}}$.
Combined with the Bernoulli number formula for $\zeta(s)$ at positive even integers, this gives
\[
\left|2.2\sum_{s=1}^{r+1}\frac{\zeta(2s)z^{2s}}{s(2\pi)^{2s}}\left(\frac{\Gamma(2+2s;2\pi w)}{(2\pi)^{2+2s}}\right)\right|
\leq 
\frac{1.1 e^{-2\pi w}}{4\pi^2}\sum_{s=1}^{r+1}\frac{\rho^{2s}B_{2s}}{s (2\pi)^{2s}(2s)!}\left(2\pi w+\left[(2s+1)!\right]^{\frac{1}{2s+1}}\right)^{2s+1}
.
\]
A simple manipulation bounds this by 
\[
\frac{1.1 \cdot e^{-2\pi w}}{4\pi^2}\sum_{s=1}^{r+1}\frac{z^{2s}B_{2s}}{s (2\pi)^{2s}(2s)!}\left(2\pi w+2s\right)^{2s+1}.
\]
Using the Bernoulli number upper bound from (24.9.8) of \cite{NIST}, recalling that $w\geq \pi N_n/4$, and noting that 
$w=\frac{\pi\cos\phi}{2\rho}<\frac{\pi}{2\rho}<\frac{\pi N_n}{2},$
this is bounded by
\begin{equation*}
\begin{aligned}
0.056\cdot &e^{-2\pi w}\sum_{s=1}^{r+1}\frac{\rho^{2s}}{s(2\pi)^{4s}}(2\pi w+2s)^{2s+1}
=0.056\cdot e^{-2\pi w}\sum_{s=1}^{r+1}\rho^{2s}\left(\frac{1}{2\pi}+\frac{s}{2\pi^2}\right)^{2s}\left(\frac{2\pi w}{s}+2\right)
\\
&\leq
0.056\cdot e^{-\frac{\pi^2N_n}2}\sum_{s=1}^{r+1}\left(\frac{s\sqrt2}{4N_n}\right)^{2s}\cdot\left(\frac{\pi^2N_n}{s}+2\right)
=n_r\cdot e^{-\frac{\pi^2 N_n}{2}}.
\end{aligned}
\end{equation*}
Therefore, if  $n\geq n_r$, then we obtain the claimed inequality for $\mathcal{I}_2.$

Finally, we turn to the bound for $\mathcal{I}_3$, which we recall is
$$
\mathcal{I}_3 := -2\sum_{s=r+2}^{\infty} \frac{\zeta(2s)z^{2s}}{s (2 \pi)^{2s}} \int_{0}^{w} \frac{y^{2s+1}}{e^{2 \pi y}-1}dy.
$$
As $\zeta(2r+4)\geq \zeta(2s)$ for all $s\geq r+2,$ we have
$$
|\mathcal{I}_3| \leq  \frac{\zeta(2r+4)}{r +2} \frac{\rho^{2r+4}}{(2 \pi)^{2r+4}} \int_{0}^{w} \frac{y^{2r+5}}{e^{2 \pi y} -1}   \sum_{s=0}^{\infty} \left(\frac{y\rho}{2 \pi}\right)^{2s} dy.
$$
Therefore, using (\ref{alphadefn}) we find that
\begin{equation*}\begin{aligned}
|\mathcal{I}_3| 
&\leq
\frac{\zeta(2r+4)}{r+2}\cdot\frac{\rho^{2r+4}}{(2 \pi)^{2r+4}}  \int_{0}^{w} \frac{y^{2r+5}}{(e^{2 \pi y}-1) \left(1-(\frac{y \rho}{2 \pi})^2\right)}dy  
\\
& <  \frac{\zeta(2r+4)}{r+2}\cdot\frac{\rho^{2r+4}}{(2 \pi)^{2r+4}} \int_0^{\infty} \frac{y^{2r+5}}{e^{2 \pi y }-1} dy 
\\
& < 
\frac{\Gamma(2r+6)\zeta(2r+4)\zeta(2r+6)\rho^{2r+4}}{(r+2)(2 \pi)^{2r+7}}
=
4\pi^3\alpha_{r+2}\rho^{2r+4}.
\end{aligned}
 \end{equation*}
 Since, $\rho \leq \frac{\sqrt{2}}{N_n},$ we obtain the claimed inequality for $\mathcal{I}_3.$
\end{proof}

\subsubsection{Proof of Proposition~\ref{MajorArc}}
We begin by recalling (\ref{important}), which asserts that
$$
-\log f(x)+\frac{A}{z^2}=
\int_{\Gamma'}\frac{t\log(1-e^{-tz})}{e^{2\pi i t}-1}dt
-\int_{\Gamma}\frac{t\log(1-e^{-tz})}{1-e^{-2\pi i t}}dt.
$$
By combining the two integrals as a single integral, Lemma~\ref{LemmaGamma} gives
\begin{equation}\label{NextStep}
\log\left( f(x)\right) =\frac{A}{z^2}+\int_0^w\frac{y\log\left(2\sin\left(\frac{yz}{2}\right)\right)}{e^{2\pi y-1}}dy+O_{\leq}\left(36N_n^2e^{-\pi^2N_n}\right).
\end{equation}
To use this formula, we employ the identity
$\sin\tau=\tau\prod_{m\geq1}\left(1-\frac{\tau^2}{m^2\pi^2}\right),$
which  implies that
\[
\log(\sin(\tau))=\log\tau-\sum_{s\geq1}\sum_{m\geq1}\frac{\tau^{2s}}{s\cdot m^{2s}\pi^{2s}}=\log\tau-\sum_{s\geq1}\frac{\zeta(2s)\tau^{2s}}{s\pi^{2s}}.
\]
Hence, for every $r\geq1,$ the integral in (\ref{NextStep}) satisfies\footnote{Wright refers to $\mathcal I_1$, $\mathcal I_2$, and $\mathcal I_3$ as $I_4$, $S_1$, and $S_2$, respectively, on page 183 of \cite{Wright}.}
\begin{equation*}
\begin{aligned}
&\mathcal{I}_1+\mathcal{I}_2+\mathcal{I}_3=\int_0^w\frac{y\log\left(2\sin\left(\frac{yz}{2}\right)\right)}{e^{2\pi y-1}}dy
\\
&\ \ \ =
2\int_0^w\frac{y\log(yz)}{e^{2\pi y}-1}dy
-2\sum_{s=1}^{r+1}\frac{\zeta(2s)z^{2s}}{s(2\pi)^{2s}}\int_0^w\frac{y^{2s+1}}{e^{2\pi y}-1}dy
-2\sum_{s\geq r+2}\frac{\zeta(2s)z^{2s}}{s(2\pi)^{2s}}\int_0^w\frac{y^{2s+1}}{e^{2\pi y}-1}dy.
\end{aligned}
\end{equation*}
Thanks to Lemma~\ref{I_integralbounds}, for $n\geq n_r,$
 (\ref{NextStep}) gives\footnote{This is an explicit form of the first equation on p.~184 of \cite{Wright}. We  correct a typographical error where the sum on $s$ accidentally starts with $s=0$ instead of $s=1.$}
\begin{equation}\label{AllTogether}
\begin{aligned}
\log f(x)
=\frac{A}{z^2}+c+\frac{\log z}{12}-\sum_{s=1}^{r+1}\alpha_sz^{2s}
+O_{\leq}\left(2^{r+4} \pi^3\alpha_{r+2} N_n^{-2r-4}+5e^{-4.7N_n}\right).
\end{aligned}
\end{equation}

We now make use of a complex-analytic version for the remainder terms  of the Taylor series of $f(x),$ which is required as our estimates make use of the expressions involving $\beta_s$ as opposed to $\alpha_s$.
Specifically, the $\beta_s$ were defined in \eqref{BetaDefn} to be the initial $r+2$ Taylor coefficients of  
\[
g_r(z):=e^{-\sum_{s=1}^{r+1}\alpha_sz^{s}}=\sum_{s=0}^{r+1}\beta_sz^{s}+R_r(z),
\]
 where $R_r(z)$ is the remainder. 
For convenience, we assume that $n\geq55$, which guarantees that
$|z|=\rho\leq\sqrt2 N_n^{-1}\leq\frac12.$
The standard complex Taylor series remainder estimate  (for example, see Theorem~B.21 of \cite{Knapp} with $R=1$)
gives
\[
|R_r(z)|\leq\frac{\max_{|z|=1}(|g_r(z)|)\cdot|z|^{r+2}}{1-|z|}\leq C_r|z|^{r+2},
\]
where
\begin{equation}\label{Cr}
C_r:=2\max_{|z|=1} \left(\left|e^{-\sum_{s=1}^{r+1}\alpha_sz^{s}}\right|\right).
\end{equation}
By replacing $z$ with $z^2$ (this is allowed since we demanded that $|z|<1$, and so $|z^2|<1$ is still in the range of validity for the remainder estimate), we obtain 
\[
g_r(z^2)=e^{-\sum_{s=1}^{r+1}\alpha_sz^{2s}}=\sum_{s=0}^{r+1}\beta_sz^{2s}+O_{\leq}\left(C_r|z|^{2r+4}\right).
\]
Therefore, by exponentiating \eqref{AllTogether}, for $n\geq 55$ we obtain
$$
f(x)=e^cz^{\frac1{12}}e^{\frac{A}{z^2}}\left(\sum_{s=0}^{r+1}\beta_sz^{2s}+O_{\leq}\left(C_r|z|^{2r+4}\right)\right)
\cdot  O_{\leq}\left(\exp(2^{r+4} \pi^3\alpha_{r+2} N_n^{-2r-4}+5e^{-4.7N_n})\right).
$$
To address the error term on the far right above, we assume that $n\geq \ell_r,$ which by \eqref{ellrdefn} gives
\[
2^{r+4} \pi^3\alpha_{r+2} N_n^{-2r-4}+5e^{-4.7N_n}\leq\frac12< 1\ \text{ and }\ \frac{1}{1-(2^{r+4} \pi^3\alpha_{r+2} N_n^{-2r-4}+5e^{-4.7N_n})}\leq2.
\]
Thanks to (4.5.11) of \cite{NIST}, which states that $e^x<1+x/(1-x)$ for $x<1$, this gives
\begin{displaymath}
\begin{split}
f(x)&=e^cz^{\frac1{12}}e^{\frac{A}{z^2}}\left(\sum_{s=0}^{r+1}\beta_sz^{2s}+O_{\leq}\left(C_r|z|^{2r+4}\right)\right)
\cdot  O_{\leq}\left(1+2^{r+5} \pi^3\alpha_{r+2} N_n^{-2r-4}+10e^{-4.7N_n}\right)\\
&=M(x)+ \mathcal{X}_r(n)+\mathcal{Y}_r(n),
\end{split}
\end{displaymath}
where we let 
$M(x):=e^{c+\frac{A}{z^2}}\sum_{s=0}^{r+1}\beta_sz^{2s+\frac{1}{12}},$
and where 
\begin{equation}\label{Er}
\mathcal{X}_r(n):=e^{c+AN_n^2}2^{r+\frac{49}{24}}C_rN_n^{-2r-\frac{49}{12}}
\end{equation}
and
\begin{equation}\label{Fr}
\mathcal{Y}_r(n):=\left|e^{c+AN_n^2}\left(2^{r+5} \pi^3\alpha_{r+2} N_n^{-2r-4}+10e^{-4.7N_n}\right)\left(2^{r+\frac{49}{24}}C_rN_n^{-2r-\frac{49}{12}}+\sum_{s=0}^{r+1}2^{s+\frac{1}{24}}\beta_sN_n^{-2s-\frac{1}{12}}\right)\right|.
\end{equation}
This encodes the compilation of error
on $C_{N_n}'$ using the facts that
$\left|e^{\frac{A}{z^2}}\right|\leq e^{\left|\frac{A}{\rho^2}\right|}\leq e^{AN_n^2}$
and $|z|=\rho\leq\sqrt 2N_n^{-1}$. 
Now, recalling that $n=2AN_n^3$, we obtain
\begin{equation*}\begin{aligned}  \label{JApprox}
J(n)&=\frac{1}{2\pi i}\int_{\frac{1-i}{N_n}}^{\frac{1+i}{N_n}}f(e^{-z})e^{2AN_n^3z}dz
=\frac{1}{2\pi i}\int_{\frac{1-i}{N_n}}^{\frac{1+i}{N_n}}(M(x)+\mathcal X_r(n)+\mathcal Y_r(n))e^{2AN_n^3z}dz
\\
&
=\frac{e^c}{2\pi i}\int_{\frac{1-i}{N_n}}^{\frac{1+i}{N_n}}\left(\sum_{s=0}^{r+1}\beta_sz^{2s+\frac{1}{12}}\right)e^{\frac{A}{z^2}+2AN_n^3z}dz
+O_{\leq}\left(\frac{(\mathcal X_r(n)+\mathcal Y_r(n))\cdot e^{2AN_n^2}}{N_n \pi}\right),
\end{aligned}
\end{equation*} 
where we used that on the path of integration, $|e^{2AN_n^3z}|=e^{2AN_n^2},$ and that the length is $2N_n^{-1}$.
We now let $v=N_n z,$ and introduce Wright's
$$
P_s:= \frac{1}{2 \pi i} \int_{1-i}^{1+i} v^{2s+\frac{1}{12}} \exp \left(AN_n^2\left(2v+\frac{1}{v^2}\right)\right)dv,
$$ 
to find that 
\begin{equation}\label{JApprox}
\begin{aligned}
J(n)&=e^c\cdot\sum_{s=0}^{r+1}\frac{\beta_sP_s}{N_n^{2s+\frac{13}{12}}}
+O_{\leq}\left(\frac{(\mathcal X_r(n)+\mathcal Y_r(n))\cdot e^{2AN_n^2}}{N_n \pi}\right).
\end{aligned}
\end{equation} 

To complete the proof we require an explicit version of Wright's expansion of $P_s$, that he obtained via the method of steepest descent.  We follow Wright in this regard. Using his notation, we first let $\mathcal C$ be the plane curve defined by the equation $(x^2+y^2)^2=x$, together with the labelled points $E=(2^{-2/3},2^{-2/3})$, $D=(2^{-2/3},-2^{-2/3})$ on $\mathcal C$ and the points $O=(0,0)$, $G=(1,1)$, and $F=(1,-1)$ in the plane. This is illustrated in the 
following figure. 
\smallskip
\begin{center}
\includegraphics[height=60mm]{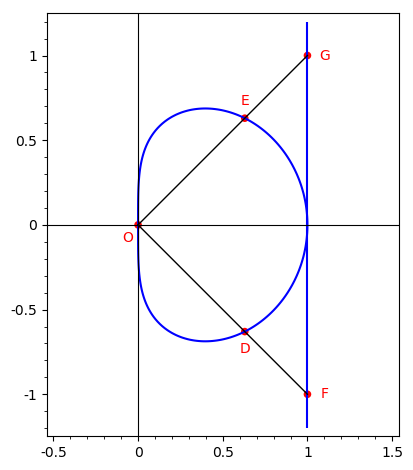}
\captionof{figure}{The curve $\mathcal C$ and the points $E,D,O,G,F$.}
\end{center}
\smallskip
%Sage code to produce this picture:
% implicit_plot((x^2+y^2)^2-x,(x,-0.5,1.5),(y,-1,1),axes='true')+implicit_plot(x-1,(x,-0.5,1.5),(y,-1.2,1.2),axes='true')+point([1,1],color='red',size=30)+point([1,-1],color='red',size=30)+point([2^(-2/3),2^(-2/3)],
%color='red',size=30)+point([2^(-2/3),-2^(-2/3)],color='red',size=30)+point([0,0],color='red',size=30)+line([(0,0), (1,1)],color='black')+line([(0,0), (1,-1)]
%,color='black')+text('O',[-0.06,-0.08],color='red')+text('E',[2^(-2/3),2^(-2/3)+0.1],color='red')+text('D',[2^(-2/3),-2^(-2/3)-0.1],color='red')+text('G',[1.1,1],color='red')+text('F',[1.1,-1],color='red')
Wright noted that if $\xi_s(v):=(2\pi i)^{-1}v^{2s+1/12}\exp(AN_n^2(2v+1/v^2))$, then since $|v|\leq\sqrt2$, $x\leq 1$ on $OG$ and $OF$, then we have
\[
|\xi_s(v)|=\frac{|v|^{2s+\frac12}}{2\pi}e^{2AN_n^2x}\leq\frac{2^{s+\frac14}e^{2AN_n^2}}{2\pi},\quad (v\in OF\cup OG),
\]
and  he cleverly showed\footnote{This corrects a typo of Wright's concerning the path of integration of the fourth integral on the right hand side. Wright accidentally wrote $\int_O^F$ instead of $\int_O^D$.} (see p. 186 of \cite{Wright}), in terms of the arc lengths in the diagram, that
\begin{equation}\label{PsEstimate}
\begin{aligned}
P_s&=\int_\mathcal C\xi_s(v)dv+O_{\leq}\left(\left|\int_D^F\xi_s(v)dv\right|+\left|\int_G^E\xi_s(v)dv\right|+\left|\int_E^O\xi_s(v)dv\right|+\left|\int_O^D\xi_s(v)dv\right|\right)
\\
&=\int_\mathcal C\xi_s(v)dv+O_{\leq}\left(\frac{2^{s+\frac14}e^{2AN_n^2}}{2\pi}\left(\len(DF)+\len(GE)+\len(EO)+\len(OD)\right)\right)
\\
&=\int_\mathcal C\xi_s(v)dv+O_{\leq}\left(0.64\cdot 2^s\cdot e^{2AN_n^2}\right).
\end{aligned}
\end{equation}
Making the change of variables $t^2=3-2v-v^{-2}$ to estimate the integral $\int_\mathcal C\xi_s(v)dv,$ we have
\[
\int_\mathcal C\xi_s(v)dv=e^{3AN_n^2}\int_\R\chi_s(t)e^{-AN_n^2t^2}dt,
\]
where 
\[
\chi_s(t)=\frac{v^{2s+\frac{37}{12}}\cdot t}{2\pi i(1-v^3)}=\frac{v^{2s+\frac{25}{12}}\sqrt{2v+1}}{2\pi(v^2+v+1)}
\]
is a smooth function on $\mathcal C.$
Expanding $\chi_s(t)=\sum_{m\geq0}a_{s,m}t^m$, we can estimate the Taylor remainder for all $t\in\mathbb R$ by 
\[
\chi_s(t)=\sum_{m=0}^{2r+3}a_{s,m}t^m+O_\leq\left(D_r\cdot |t|^{2r+4}\right),
\]
where 
\begin{equation}\label{Dr}
D_r:=\frac{1}{(2r+4)!}\cdot\max\left\{
\max\left\{\left|\chi_s^{(2r+4)}(t)\right|\right\}_{t\in\R}
\right\}_{s=0}^{r+1}.
\end{equation}
We note that $D_r$ is  explicitly computable, as $\chi_s(t)$ is a smooth function on the compact curve $\mathcal C.$ 
Thus, we find  that\footnote{This corrects a typo in Wright's work, where he drops the factor $e^{3AN_n^2}$ temporarily in the final displayed equations of page 187.}
\begin{equation*}
\begin{aligned}
\int_\mathcal C\xi_s(v)dv
& =e^{3AN_n^2}\sum_{m=0}^{2r+3}\int_\R a_{s,m}t^me^{-AN_n^2t^2}dt
+O_{\leq}\left(D_r\cdot e^{3AN_n^2}\int_\R |t|^{2r+4}e^{-AN_n^2t^2}dt\right)
\\
&=e^{3AN_n^2}\sum_{m=0}^{2r+3}\int_\R a_{s,m}t^me^{-AN_n^2t^2}dt+O_{\leq}\left(D_r\cdot\Gamma\left(r+\frac52\right)(AN_n^2)^{-\frac52-r}e^{3AN_n^2}\right).
\end{aligned}
\end{equation*}
Plugging into \eqref{PsEstimate} gives
\begin{equation*}
\begin{aligned}
P_s
=e^{3AN_n^2}\sum_{m=0}^{r+1}\frac{a_{s,2m}\Gamma\left(m+\frac12\right)}{(AN_n^2)^{m+\frac12}}+O_{\leq}\left(D_r\cdot\Gamma\left(r+\frac52\right)(AN_n^2)^{-\frac52-r}e^{3AN_n^2}+0.64\cdot2^{s}e^{2AN_n^2}\right).
\end{aligned}
\end{equation*}
Finally, Wright proved that 
$a_{s,2m}=(-1)^mb_{s,m}/2\pi ,$
 for $s\leq r+1,$ and so we obtain
\begin{equation*}\label{PsFullExpansion}
P_s=e^{3AN_n^2}\sum_{m=0}^{r+1}\frac{(-1)^mb_{s,m}\Gamma\left(m+\frac12\right)}{2\pi\cdot(AN_n^2)^{m+\frac12}}+O_{\leq}\left(D_r\cdot\Gamma\left(r+\frac52\right)(AN_n^2)^{-\frac52-r}e^{3AN_n^2}+0.64\cdot2^{r+1}e^{2AN_n^2}\right).
\end{equation*}
Plugging this into \eqref{JApprox} gives
\begin{equation}\label{FinalApprox}
\begin{aligned}
J(n)&=\frac{e^{c+3AN_n^2}}{2\pi }\sum_{s=0}^{r+1}\sum_{m=0}^{r+1}\frac{(-1)^m\beta_sb_{s,m}\Gamma\left(m+\frac12\right)}{A^{m+\frac12}N_n^{2s+2m+\frac{25}{12}}}
+
O_{\leq}\Bigg(
\frac{(\mathcal X_r(n)+\mathcal Y_r(n))e^{2AN_n^2}}{N_n\pi}
\\ & +\left|e^c\left(D_r\cdot\Gamma\left(r+\frac52\right)(AN_n^2)^{-\frac52-r}e^{3AN_n^2}+0.64\cdot2^{r+1}e^{2AN_n^2}\right)\sum_{s=0}^{r+1}\beta_sN_n^{-2s-\frac{13}{12}}\right|
\Bigg).
\end{aligned}
\end{equation}
Therefore, the proof is complete by letting
\begin{equation}\label{ErMaj}
\widehat{E}_r^{\maj}(n):=\frac{(\mathcal X_r(n)+\mathcal Y_r(n))e^{2AN_n^2}}{N_n \pi} +|\mathcal{Z}_r(n)|,
\end{equation}
where
\begin{equation}\label{Zr}
\mathcal{Z}_r(n):=e^c\left(D_r\cdot\Gamma\left(r+\frac52\right)(AN_n^2)^{-\frac52-r}e^{3AN_n^2}+0.64\cdot2^{r+1}e^{2AN_n^2}\right)\sum_{s=0}^{r+1}\beta_sN_n^{-2s-\frac{13}{12}}.
\end{equation}

\subsection{Proof of Theorem~\ref{EffectiveWright}}
Cauchy's theorem (\ref{Cauchy}) gives
$\PL(n) = J(n)+ E^{\min}(n).$
Therefore, the theorem follows from Proposition~\ref{ZagierHelp} and Proposition~\ref{MajorArc}.

\section{Proof of Theorems~\ref{Thm1} and \ref{Thm2}}\label{ProofsPart2}
Here we make use of Theorem~\ref{EffectiveWright} to prove Theorems~\ref{Thm1} and \ref{Thm2}.

\subsection{Proof of Theorem~\ref{Thm1}}\label{ProofsPart1}

To apply Theorem~\ref{EffectiveWright} with a fixed $r,$ it is natural  to express 
\begin{equation}
\PL(n)=\widehat{\PL}_r(n)+O_{\leq}\left(\mathcal E_r(n)\right),
\end{equation}
where $\widehat{\PL}_r(n)$ is the main term and $\mathcal E_r(n)$ is an explicit bound for the error. Then we can write
\begin{equation}\label{LCEst}
\begin{split}
\PL(n)^2-&\PL(n-1)\PL(n+1) \\
&\ \ \ \ \  \geq (\widehat{\PL}_r(n)-\mathcal E_r(n))^2-(\widehat{\PL}_r(n-1)+\mathcal E_r(n-1))\cdot(\widehat{\PL}_r(n+1)+\mathcal E_r(n+1)).
\end{split}
\end{equation}

We apply Theorem~\ref{EffectiveWright} with $r=2.$ 
We first determine the $n$ to which it applies.
By (\ref{nrdefn}),  $n_2$ is the point beyond where the following expression is guaranteed to be less than $1$:
\[
\frac{7 \, {\left(243 \cdot 2^{\frac{2}{3}} \pi^{2} n^{\frac{1}{3}} A^{\frac{5}{3}} + 32 \cdot 2^{\frac{1}{3}} \pi^{2} n^{\frac{5}{3}} A^{\frac{1}{3}} + 64 \pi^{2} n A + 64 \cdot 2^{\frac{2}{3}} n^{\frac{4}{3}} A^{\frac{2}{3}} + 256 \cdot 2^{\frac{1}{3}} n^{\frac{2}{3}} A^{\frac{4}{3}} + 2916 A^2\right)}}{32000 n^{2}}.\]
This expression is decreasing in $n.$ At $n=1$ it is $\approx2.4$, and at $n=2$ it is $\approx0.8$, and so $n_2=2$. 
To compute $\ell_2$, we determine when (\ref{ellrdefn}) is guaranteed to be less than $1/2$. This quantity is bounded from above by
$0.00005n^{-\frac83}+5e^{-3.5n^{\frac13}},$
which is decreasing in $n$ and is less than $0.16$ for $n=1$, and so we have $\ell_2=1$. 
Therefore, $\max\{\ell_2,n_2,87\}=87$, and so Theorem~\ref{EffectiveWright} holds for $n\geq87$. 

The terms defining $\widehat{\PL}_r(n)$ come from the double sum in Theorem~\ref{EffectiveWright}, and can be organized by the powers of $n$ which are controlled by $s+m$. To this end, we recall
\[
\PL(n)\approx\frac{e^{c+3AN_n^2}}{2\pi }\sum_{s=0}^{r+1}\sum_{m=0}^{r+1}\frac{(-1)^m\beta_sb_{s,m}\Gamma\left(m+\frac12\right)}{A^{m+\frac12}N_n^{2s+2m+\frac{25}{12}}}=:e^{3AN_n^2}n^{-\frac{25}{36}}\sum_{s=0}^{r+1}\sum_{m=0}^{r+1}f_{s,m}(n).
\]
The leading asymptotic is given by 
$f_{0,0}(n)=2^{\frac{25}{36}}e^cA^{\frac{7}{36}}/\sqrt{12\pi}.$
The next largest term, with a power saving of $n^{-\frac23}$, is given by the terms with $s+m=1$:
\[
f_{1,0}(n)+f_{0,1}(n)=-\frac{\sqrt{3}\cdot2^{\frac{13}{36}}e^{c} {\left(3 A + 1385\right)} }{25920 \, \sqrt{\pi} A^{\frac{5}{36}}}\cdot n^{-\frac23}.
\]
The terms with $s+m=2$ give 
\[
f_{1,1}(n)+f_{2,0}(n)+f_{0,2}(n)=-\frac{\sqrt{3} \cdot2^{\frac{1}{36}} e^c{\left(1377 A^{2} - 370650  A+ 12525625\right)}}{1567641600  \sqrt{\pi} A^{\frac{17}{36}}}\cdot n^{-\frac{4}{3}}.
\]
The final term which contributes to our main estimate is the sum of terms with $s+m=3$, giving 
\begin{equation*}
\begin{aligned}
f_{3,0}(n)&+f_{0,3}(n)+f_{2,1}(n)+f_{1,2}(n)
\\&=-\frac{\sqrt{3}\cdot 2^{\frac{25}{36}}e^c {\left(609309 A^{3} - 90985275 A^{2} + 4957761375 A + 37576109375\right)}}{40633270272000  \sqrt{\pi} A^{\frac{29}{36}}}\cdot n^{-2}.
\end{aligned}
\end{equation*}
These contributions together give our main term 
\begin{equation}\label{Mainr2}
\begin{aligned}
\widehat{\PL}_2(n):=e^{3AN_n^2}n^{-\frac{25}{36}}&(f_{0,0}(n)+f_{1,0}(n)+f_{0,1}(n)+f_{1,1}(n) +f_{2,0}(n)+f_{0,2}(n)\\ +&f_{3,0}(n)+f_{0,3}(n)+f_{2,1}(n)+f_{1,2}(n))
\\
\approx &e^{c+3AN_n^2}n^{-\frac{25}{36}}(0.23-0.056 n^{-\frac23}-0.006n^{-\frac43}-0.001n^{-2}).
\end{aligned}
\end{equation}

The remaining terms with $s+m\geq4$ are of equal or smaller magnitude than the error which will come from $E^{\rm{maj}}_2(n)$ and $E^{\rm{min}}(n)$. Thus, we will put these terms into our error estimate $\mathcal E_2(n)$. A simple manipulation shows that for $n\geq87$, we have
\[
f_{2,2}(n)+f_{3,1}(n)+f_{1,3}(n)+f_{3,2}(n)+f_{2,3}(n)+f_{3,3}(n)=O_{\leq}(10^{-5}n^{-\frac73}).
\]
Thus, Theorem~\ref{EffectiveWright} implies that for $n\geq87$, we have
\begin{equation}\label{PLBound}
\PL(n)=\widehat{\PL}_2(n)+O_{\leq}\left(10^{-5}e^{3AN_n^2}n^{-\frac{109}{36}}+\widehat{E}_2^{\rm{maj}}(n)+E_2^{\rm{min}}(n)\right).
\end{equation}

We now aim to bound $\widehat{E}_2^{\rm{maj}}(n)$. This error is $O(e^{3AN_n^2}n^{-\frac{109}{36}})$.\footnote{At the top of page 189, Wright makes a claim about the magnitude of this error term which is incorrect in the power of $N_n$. This can be seen by comparing with the error estimate in the final equation of page 188, which is correct and matches our power $-109/12$ here for $r=2$.}  Thus, we will compare our error terms to this expression. We now turn to computing the relevant constants in turn.
After computing $\alpha(1)$, $\alpha(2)$, and $\alpha(3)$, we see that computing the constant $C_2$ is equivalent to maximizing 
\[\left|e^{-\frac{e^{it}}{2880}-\frac{e^{2it}}{725760}-\frac{e^{3it}}{43545600}}\right|^2.\]
The derivative of this function is
\[
\frac{1}{7257600} \, {\left(4 \, \cos\left(t\right)^{2} + 80 \, \cos\left(t\right) + 5039\right)} e^{\left(-\frac{1}{5443200} \, \cos\left(t\right)^{3} - \frac{1}{181440} \, \cos\left(t\right)^{2} - \frac{5039}{7257600} \, \cos\left(t\right) + \frac{1}{362880}\right)} \sin\left(t\right),
\]
which shows that there is a local maximum at $t=\pi$, which is the global maximum. This directly gives $C_2=2\cdot e^{\frac{15061}{43545600}}\leq 2.0007$. 

The computation of the constants $D_r$ is more involved. 
To compute $D_r$, one recursively computes the derivatives $v^{(m)}(t)=v^{(m)}$ for $m=1,\ldots,r+3$ by repeatedly differentiating the equation $t^2=3-2v-v^{-2}$. Using the definition of $\chi_s(t)$ and the relation $t=-i(v-1)\sqrt{2v+1}/v$ gives an expression for $\chi_s(t)$ as a function of $v$. Here we need $D_2$, which requires the derivatives
\begin{equation*}
\begin{aligned}
v'&=tv^3/(v^3-1),\\ 
v''&=v^3 (1 + 3 t^2 v^2 - 2 v^3 + v^6)/(v^3-1)^3,  \\
  v^{(3)}&=3tv^5 (3  + 5 t^2 v^2 - 6 v^3 + 4 t^2 v^{5} + 3 v^{6})/(1-v^3)^5, 
 \\ v^{(4)}&=3 v^5 (3 + 30 t^2 v^2 - 12 v^3 +\dots
   + 20 t^4 v^{10}   
   + 24 t^2 v^{11} + 
   3 v^{12})/(1-v^3)^7,
   \\ v^{(5)}&=15tv^7(15 + 70 t^2 v^2 - 48 v^3 +\dots
  + 24 t^4 v^{13} + 40 t^2 v^{14} + 12 v^{15})/(1-v^3)^9,
 \\ v^{(6)}&=47v^7(5 + 105 t^2 v^2 - 26 v^3 +\dots
  + 120 t^4 v^{19} + 60 t^2 v^{20} + 
 4 v^{21})/(1-v^3)^{11},
 \\
 v^{(7)}&=315tv^9(35 + 315 t^2 v^2+\ldots+ 120 t^2 v^{23} + 20 v^{24})/(1-v^3)^{13},
 \\
 v^{(8)}&=315v^9(35 + 1260 t^2 v^2+\ldots+ 480 t^2 v^{29} + 20 v^{30})/(1-v^3)^{15}.
\end{aligned}
\end{equation*}
Differentiating and plugging into the definition of $\chi_s(t)$, we obtain expressions such as
\begin{equation}\label{ChiEx}
\begin{aligned}
\chi_1^{(6)}(t)=
\frac{iv^{\frac{121}{12}}\sqrt{2v+1}}{859963392\pi (v^2+v+1)^{17}}\cdots (181387629768625 +\ldots + 
  2444688400 v^{20}).
\end{aligned}
\end{equation}
Now that these are expressions of $v$, we can substitute $v=x\pm i\sqrt{\sqrt{x}-x^2}$, which traces out the curve $\mathcal C$ as $x$ ranges from $0$ to $1$. That is, expressions such as the absolute value of  \eqref{ChiEx} can be evaluated for $x\in[0,1]$ numerically to obtain an estimate for $D_2$. By symmetry, we only need to bound
for $v=x+i\sqrt{\sqrt{x}-x^2}$ with $x\in[0,1]$. Performing this analysis, we find that $D_2\leq 5.3$.

These constants allow us to bound $\widehat{E}^{\rm{maj}}_2(n)$ by \eqref{ErMaj} for $n\geq87.$ A simple calculation gives
\[
\frac{\mathcal X_2(n)+\mathcal Y_2(n)}{N_n\pi }\cdot e^{2AN_n^2}=O_{\leq}\left((127-10^{-5})e^{3AN_n^2}n^{-\frac{109}{36}}\right).
\]
Further simplification gives (see \eqref{Zr})
$\mathcal{Z}_2(n)=O_{\leq}\left(100e^{3AN_n^2}n^{-\frac{109}{36}}\right).$
Therefore, we find that
$$
\widehat{E}^{\rm{maj}}_2(n)=O_{\leq}\left((227-10^{-5})e^{3AN_n^2}n^{-\frac{109}{36}}\right).
$$
Since  we have
$E^{\rm{min}}(n)=O_{\leq}\left(e^{\left(3A-\frac25\right)N_n^2}\right),$ expression
 \eqref{PLBound} yields  
\begin{equation}\label{r2estimate}
\PL(n)=\widehat{\PL}_2(n)+O_{\leq}\left(227e^{3AN_n^2}n^{-\frac{109}{36}}+e^{\left(3A-\frac25\right)N_n^2}\right).
\end{equation}
In other words, we can let 
$\mathcal{E}_2(n):=227e^{3AN_n^2}n^{-\frac{109}{36}}+e^{\left(3A-\frac25\right)N_n^2}.$
A simple calculation with \eqref{LCEst} and  (\ref{Mainr2}) establishes  that $\PL(n)$ is log-concave for all $n\geq8820.$
This completes the proof as log-concavity has been confirmed on a computer for all $12\leq n\leq 10^5$
by Heim et al. \cite{HNT}.

\subsection{Proof of Theorem~\ref{Thm2}}

The proof of Theorem~\ref{Thm2} makes use of Theorem~\ref{EffectiveWright} and recent work by Griffin, Zagier, and two of the authors of \cite{GORZ} on Jensen polynomials of suitable sequences of real numbers. The main observation is that suitable real sequences 
have Jensen polynomials that can be
 modeled by the {\it Hermite polynomials} $H_d(X)$, which
are orthogonal polynomials for the measure $\mu(X)=e^{-X^2/4},$ and are  given by
the generating function
$$
{\sum_{d=0}^{\infty} H_d(X)\,\frac{t^d}{d!} \= e^{-t^2+Xt}} \= 
  1\ +X\,t  \+ (X^2-2)\,\frac{t^2}{2!}\+ (X^3-6X)\,\frac{t^3}{3!}\+\cdots
$$

\begin{theorem}[Theorems 3 and 6 of \cite{GORZ}]\label{GeneralTheorem2}
Let $\{\alpha(n)\}, \{A(n)\},$ and $\{\delta(n)\}$ be sequences of positive real numbers, with $\delta(n)$ tending to $0.$ For an integer $d\geq 3,$ suppose that there are real numbers $g_3(n), g_4(n),\dots g_d(n),$ for which
 \begin{equation}\label{log_alpha_ratio}
{ \log \Bigl({\frac{\alpha(n+j)}{\alpha(n)}\Bigr) \, \= \, A(n)j-\delta(n)^2 j^2+\sum_{i=3}^{d} g_i(n)j^i
  + \text{\rm o}\bigl(\delta(n)^d\bigr)}} \qquad\text{\rm as $n\to\infty$},  \end{equation}
with $g_i(n)=o(\delta(n)^i)$ for each $3\leq i\leq d.$ Then we have
\begin{equation}\label{tag2} \lim_{n\to\infty} \biggl(\frac{\delta(n)^{-d}}{\alpha(n)}
  \,J_\alpha^{d,n}\Bigl(\frac{\delta(n)\,X\m1}{\exp(A(n))}\Bigr)\biggr) \= H_d(X).
  \end{equation}
  \end{theorem}
 \begin{remark}
  Theorem~\ref{GeneralTheorem2} holds for $d\in \{1, 2\}.$ In these cases there simply are no numbers $g_i(n)$. In fact, one can obtain Theorem~\ref{Thm2} for $d=2$ using this result, which is essentially the method employed by Heim et. al. in \cite{HNT} without  the formalism of Jensen polynomials.
 \end{remark}
  
Since the Hermite polynomials are hyperbolic, and since this property of a polynomial with real
coefficients is invariant under {small} deformation, we have the following key consequence.

\begin{corollary*}\label{HyperbolicCorollary}
Assuming the hypotheses in Theorem~\ref{GeneralTheorem2}, we have that
the Jensen polynomials $J_{\alpha}^{d,n}(X)$ are hyperbolic for all but finitely many 
values $n.$
\end{corollary*}

\begin{proof}[Proof of Theorem~\ref{Thm2}]
Theorem~\ref{Thm1} implies that $J_{\PL}^{2,n}(X)$ is hyperbolic for all $n\geq 12.$
Therefore, without loss of generality we may assume that $d\geq 3$.
To complete the proof, we apply Theorem~\ref{GeneralTheorem2}.

Theorem~\ref{EffectiveWright} gives constants $\kappa>0$ and $\nu_0, \nu_1,\dots$ for which $\PL(n)$ has an asymptotic expansion to all orders of $1/n$ of the form
$$
\PL(n)\sim \exp\left(\sqrt[3]{\kappa n^2}\right) n^{-\frac{25}{36}}\cdot \left(\nu_0-\sum_{m=1}^{\infty}\frac{\nu_m}{n^{\frac{2m}{3}}}\right).
$$
We choose $\nu_0$  so that the leading factor is a product of an exponential with a power of $n.$ 
Using $\log(1-X)=-X-X^2/2-X^3/3-X^4/4-\dots,$ we can rewrite this expression as
$$
\PL(n)\sim \exp(\sqrt[3]{\kappa n^2}) n^{-\frac{25}{36}}\cdot \exp\left(c_0+\frac{c_1}{n^{2/3}}+\frac{c_2}{n^{4/3}}+\frac{c_3}{n^2}+\dots\right),
$$ 
which in turn gives
$$
\log\left(\frac{\PL(n+j)}{\PL(n)}\right)\sim
\sqrt[3]{\kappa}\sum_{i=1}^{\infty}
\binomial{2/3}{i}\frac{j^i}{n^{i-\frac{2}{3}}}-\frac{25}{36}
\sum_{i=1}^{\infty}\frac{(-1)^{i-1}j^i}{i n^i}+
\sum_{s,t\geq 1} c_t \binomial{-t}{s}\frac{j^s}{n^{2(s+t)/3}}.
$$
As $\binomial{2/3}{1}=2/3>0$ and $\binomial{2/3}{2}=-1/9<0,$ Theorem~\ref{GeneralTheorem2} applies, and so its corollary
proves the theorem.
\end{proof}


\begin{thebibliography}{BrStr}

\bibitem{Almkvist} G. Almkvist, \emph{A rather exact formula for the number of plane partitions},  A tribute to Emil Grosswald: number theory and related analysis, Contemp. Math. {\bf 143}, Amer. Math. Soc., 1993, 1-26.

\bibitem{Andrews} G. E. Andrews, \emph{The theory of partitions}, Cambridge Univ. Press, Cambridge, 1998.

\bibitem{CFT} K. Bringmann, B. Kane, L. Rolen, and Z. Tripp, \emph{Fractional partitions and conjectures of Chern­­­­­­­­­­­­-Fu-Tang and Heim-Neuhauser}, Trans. Amer. Math. Soc., Ser. B,  \textbf{8} (2021), 615-634.

\bibitem{ChenJiaWang} W. Y. C. Chen, D. X. Q. Jia, and L. X. W. Wang, \emph{Higher Order Tur\'an Inequalities for the Partition Function}, Trans. Amer. Math. Soc. \textbf{372} (2019), 2143-2165.

\bibitem{DDMP} A. Dabholkar, F. Denef, G. Moore, and B. Pioline,
\emph{Precision counting of small black holes,} J. High Energy Physics (2005), Art. 096.

\bibitem{DP} S. Desalvo and I. Pak, \emph{Log-concavity of the partition function},
Ramanujan J. \textbf{38} (2015), 61-73.

\bibitem{GP} S. Govindarajan and N. Prabhakar, \emph{A superasymptotic formula for the number of plane partitions},
(\texttt{https://arxiv.org/pdf/1311.7227.pdf}), preprint.


\bibitem{GORZ} M. Griffin, K. Ono, L. Rolen, and D. Zagier,
\emph{Jensen polynomials for the Riemann zeta function and other sequences},
Proc. Natl. Acad. Sci. USA \textbf{116} (2019), no. 23, 11103-11110.

\bibitem{HNT} B. Heim, M. Neuhauser, and R. Tr\"oger,
\emph{Inequalities for plane partitions}, (arXiv link: \texttt{https://arxiv.org/abs/2109.15145}), preprint.

\bibitem{Knapp} A. W. Knapp, {\it Basic Real Analysis} Appendix B. Elementary Complex Analysis, Books By Independent Authors, 2016: 631-714 (2016). 

\bibitem{LarsonWagner} H. Larson and I. Wagner, \emph{Hyperbolicity of the partition Jensen
polynomials}, Res. Numb. Th. \textbf{5} (2019), Art. 19.

\bibitem{MacMahon} P. A. MacMahon, \emph{Combinatory analysis, Vol. I, II}, Dover Publ., Mineola, New York, 2004.


\bibitem{Nicolas} J.-L. Nicolas, \emph{Sur les entiers $N$ pour lesquels il y a beaucoup de groupes ab\'eliens d'order $N$}, Ann. Inst. Fourier \textbf{28} (1978), 1--16.

\bibitem{NIST} NIST Digital Library of Mathematical Functions. http://dlmf.nist.gov/, Release 1.1.3 of 2021-09-15. F. W. J. Olver, A. B. Olde Daalhuis, D. W. Lozier, B. I. Schneider, R. F. Boisvert, C. W. Clark, B. R. Miller, B. V. Saunders, H. S. Cohl, and M. A. McClain, eds.

\bibitem{Pandey} B. Pandey, \emph{Ph.D.  Thesis}, University of Virginia, 2023.

%\bibitem{Pinelis} I. Pinelis, \emph{Exact lower and upper bounds on the incomplete gamma function}, arXiv:2005.06384. 

 
\bibitem{Stanley} R. Stanley, \emph{Log-concave and unimodal sequences in algebra, combinatorics, and geometry}, Graph theory and its applications: East and West (Jinan, 1986), 500-535,  Ann. New York Acad. Sci., 576, New York Acad. Sci., New York, 1989. 


\bibitem{Wright} E. M. Wright, \emph{Asymptotic partition formulae, I: Plane partitions},
Quarterly J. Math. Oxford, \textbf{2}, (1931), 177-189.

\bibitem{Zagier} D. Zagier, \emph{The Mellin transform and other useful analytic techniques}, Appendix to E. Zeidler, Quantum field theory I: Basics in Mathematics and Physics. A Bridge Between Mathematicians and Physicists, Springer-Verlag, New York (2006), 305-323.

\end{thebibliography}
\end{document}